\documentclass[10pt]{amsart}
\usepackage[utf8]{inputenc}
\usepackage{amsmath}
\usepackage{amssymb}
\usepackage{enumerate}

\makeatletter
\@namedef{subjclassname@2010}{
 \textup{2010} Mathematics Subject Classification}
\makeatother

\newtheorem{Thm}{Theorem}[section]
\newtheorem{Prp}[Thm]{Proposition}
\newtheorem{Cor}[Thm]{Corollary}
\newtheorem{Lma}[Thm]{Lemma} 
\theoremstyle{definition}
\newtheorem{Def}[Thm]{Definition}
\newtheorem{Rk}[Thm]{Remark}
\newtheorem{Ex}[Thm]{Example}

\numberwithin{equation}{section}

\newcommand{\Stab}{\operatorname{Stab}}
\newcommand{\sgn}{\operatorname{sgn}}
\newcommand{\Aut}{\operatorname{Aut}}
\newcommand{\Id}{\operatorname{Id}}
\newcommand{\Der}{\operatorname{Der}}

\newcommand{\ch}{\operatorname{char}}

\newcommand{\Z}{\mathbb Z}

\newcommand{\R}{\mathbb R}
\newcommand{\C}{\mathbb C}
\newcommand{\HH}{\mathbb H}
\newcommand{\OO}{\mathbb O}
\newcommand{\A}{\mathbb A}
\newcommand{\T}{\mathbb S^3}

\begin{document}
\title[Division Composition Algebras with Non-Abelian Derivation Algebras]{Classification of the Finite-Dimensional Real
Division Composition Algebras having a Non-Abelian Derivation Algebra}
\author[S. Alsaody]{Seidon Alsaody}
\address{Uppsala University\\Dept.\ of Mathematics\\P.O.\ Box 480\\751 06 Uppsala\\Sweden}
\email{seidon.alsaody@math.uu.se}

\begin{abstract} We classify the category of finite-dimensional real division composition algebras
having a non-abelian Lie algebra of derivations. Our complete and explicit classification is largely achieved by introducing
the concept of a quasi-description of a category, and using it to express the problem in terms of normal form problems for certain
group actions
on products of 3-spheres.
\end{abstract}

\maketitle
\section{Introduction}
Division algebras and composition algebras constitute two important classes of not necessarily associative algebras
which are defined over any field. Over the real numbers, finite-dimensional algebras from either class exist exclusively in
dimensions 1, 2, 4 and 8. For division algebras, this was proven in the classical work of Hopf \cite{H}, Bott--Milnor \cite{BM},
and Kervaire \cite{Ke}. For composition algebras, this is true over any field, and was proven by Kaplansky
\cite{Ka} in 1953. Despite their long history, the problem of classifying these algebras is still unsolved, even
under the assumption of the ground field being $\R$ and the dimension being finite. 

In this study we consider finite-dimensional real algebras which are both division algebras and composition algebras. These
algebras are characterized by the property that they are absolute valued, i.e.\ equipped with a multiplicative norm. The fact
that these algebras only exist in dimensions 1, 2, 4 and 8 was proven by Albert \cite{A} already in 1947. While a classification
exists in dimension at most four, the case of dimension eight is far from fully understood, and has exhibited
difficulties.

One way to approach these classes of algebras is by considering those objects which exhibit a high degree of symmetry, in the
sense of having a large automorphism group. Examples are the algebras of the octonions, para-octonions, and the Okubo algebras,
which, as the references indicate, have many interesting connections to various algebraic objects, as well as applications in
mathematics and physics. For a real division algebra $A$, the automorphism group is a Lie group, its Lie
algebra being the derivation algebra $\Der(A)$, and $A$ carries the structure of a $\Der(A)$-module. This allows for tools
from representation theory to be used. The approach to finite-dimensional real division algebras via their derivation algebras was
taken by Benkart and Osborn in \cite{BO1} and \cite{BO2}, where they determined which Lie algebras may occur as such derivation
algebras, and how the division algebras decompose as modules over their derivation algebras. For
division algebras with large derivation algebras, they determined the multiplication tables as well. A similar study, using
automorphism groups instead of derivation algebras, was conducted in \cite{DZ}. Turning to finite-dimensional composition algebras
over more general fields, several investigations were carried out by Elduque, Elduque--Myung and Elduque--P\'erez, focusing mainly
on algebras with large derivation algebras (see e.g.\ \cite{EM},\cite{EP}). In the extensive paper \cite{P}, P\'erez gave an
explicit account of all finite-dimensional division composition algebras with non-abelian derivation algebras over any field $k$
of characteristic different from 2 and 3. Given a non-abelian Lie algebra $L$ occurring as such a derivation algebra, he expressed
all finite-dimensional division composition $k$-algebras $A$ with $\Der(A)=L$ as isotopes of unital composition algebras, and
determined their decomposition into irreducible $L$-modules.

With the results from \cite{P} as a starting point, we work toward classifying the finite-dimensional real division composition
algebras having a non-abelian derivation algebra. The paper is organized
as follows. After the necessary preliminaries, we recall the classification in the case of dimension at most 4. In Section 2 we
review some general results on eight-dimensional real division composition algebras in preparation for our further study. From
\cite{P} we
know that if $A$ is an eight-dimensional division composition algebra over a field of characteristic not 2 or 3, having a
non-abelian derivation algebra, and if
$A=A_1\oplus\cdots\oplus A_n$ is a decomposition of $A$ into irreducible $\Der(A)$-modules, then the set $\{\dim A_i|1\leq i\leq
n\}$ is invariant under isomorphism, and attains one of eight possible values. Accordingly, the category $\mathcal D$ of
eight-dimensional real division composition algebras having a non-abelian derivation algebra decomposes into eight blocks. In
Section 3, we give a description (in the sense of Dieterich) of five of these, thus obtaining, in each case, an equivalence of
categories from a group action groupoid to the block at hand. 

The descriptions of Section 3 give precise information about the structure of the blocks, and extracting a classification out of
them would be technical. Aiming at a more transparent approach, we introduce in Section 4 the notion of a
quasi-description of a category. Like descriptions, a quasi-description of a category transfers the problem of classifying it to
the normal form problem of a group action. Applying this to $\mathcal D$, we obtain two quasi-descriptions, together covering six
of the eight blocks. In both
quasi-descriptions, the group actions are induced by the action of the group $\Aut(\HH)\simeq SO_3$ on the set $\T=\mathbb S(\HH)$
of quaternions of norm one. This gives a unified approach which also renders the
classification problem quite feasible. In Section 5 we
treat the remaining two blocks, achieving an explicit classification of each. Our final Section 6 is devoted to the solution of
the normal form problem for the actions involved in the aforementioned quasi-descriptions, thereby completing the classification
of $\mathcal D$.

\subsection*{Acknowledgements} I am grateful to Professor Alberto Elduque for proposing this direction of study, and for
valuable discussions, especially on the topic of Section 5.2, where his insight and his contribution of Proposition \ref{P:
AE} proved essential.

\subsection{Preliminaries}
By an \emph{algebra} over a field $k$ we understand a $k$-vector space $A$ with a bilinear multiplication, written as
juxtaposition, neither assumed to be associative or commutative, nor to admit a unity. A non-zero $k$-algebra $A$ is called a
\emph{division algebra} if for each $a\in A\setminus\{0\}$, the operators $L_a=L_a^A$ and $R_a=R_a^A$ of left and right
multiplication by
$a$, respectively, are bijective. If $A$ is a division algebra, then it has no zero divisors, and the converse is true if
the dimension of $A$ is finite. A $k$-algebra $A$ is called a \emph{composition algebra} if $A$ is equipped with a
non-degenerate\footnote{A quadratic form $q$ on a $k$-vector space $V$, where $\ch(k)\neq 2$, is called
\emph{non-degenerate} if no $x\in V\setminus \{0\}$ is orthogonal to the entire space $V$ with respect to the bilinear form
$b_q(x,y)=q(x+y)-q(x)-q(y)$. In characteristic two the definition is more involved; see e.g.\ \cite{KMRT}.}
quadratic form $q$ that is multiplicative, i.e.\ such that
\[q(xy)=q(x)q(y)\]
for all $x,y\in A$. Finally, a non-zero real algebra $A$ is called
an \emph{absolute valued algebra} if it is endowed with a
multiplicative norm.

In the setting of finite-dimensional real algebras, the norm of an absolute valued algebra and the quadratic form of a
composition algebra are each uniquely determined by the multiplication of the algebra. Thus we may speak of \emph{the} norm of
an absolute valued algebra and \emph{the} quadratic form of a composition algebra. If a composition algebra $A$ is moreover
unital, it is equipped with an involution $J: A\to A$ fixing the unity and acting as $-1$ on its orthogonal
complement. We will refer to this as the \emph{standard involution on $A$}.

The following result expands on some of the
remarks in the introduction.

\begin{Prp}\label{P: Bridge} Let $A$ be a finite-dimensional, non-zero real algebra. Then the following statements are equivalent.
\begin{enumerate}[(i)]
\item $A$ is absolute valued.
\item $A$ is a division composition algebra.
\item $A$ is a composition algebra and the quadratic form $q$ of $A$ is positive definite.
\item $A$ is a composition algebra and the quadratic form $q$ of $A$ is definite.
\end{enumerate}
\end{Prp}
We call a quadratic form \emph{definite} if it is positive definite or negative definite.

The above result is not hard to prove. For instance, (iii) is equivalent to (iv) since if
$q(a)<0$ for some $a\in A$, then $q(a^2)=q(a)^2>0$, whence $q$ cannot be negative definite. The equivalence of the first three
items is discussed in \cite{DDiss}. There, it is derived from the following important result from \cite{A}.

\begin{Prp}\label{P: Albert} Every unital finite-dimensional absolute valued algebra is isomorphic to some $\mathbb A \in
\{\mathbb R,\mathbb C,\mathbb H,\mathbb O\}$. Any finite-dimensional absolute valued algebra is isomorphic to an
orthogonal isotope $\A_{f,g}$ of
some $\mathbb A \in \{\mathbb R,\mathbb C,\mathbb H,\mathbb O\}$, i.e.\ $A=\mathbb A$ as a vector space, and the
multiplication $\cdot$ in $A$ is given by
\[x \cdot y = f(x)g(y)\] for all $x,y \in A$, where $f$ and $g$ are linear orthogonal operators on $A$, and juxtaposition denotes
multiplication in $\mathbb A$. Moreover, the norm in $\A_{f,g}$ coincides with the norm in $\A$.
\end{Prp}

We denote by $\mathcal A$ the category of all finite-dimensional absolute valued algebras (hence, of all finite-dimensional real
division composition algebras), where the morphisms are the non-zero algebra homomorphisms. It follows from the above result that
the object class of $\mathcal A$ is partitioned as
\[\mathcal A=\mathcal A_1\cup\mathcal A_2\cup\mathcal A_4\cup\mathcal A_8,\]
where $\mathcal A_d$ is the full subcategory of $\mathcal A$ consisting of all $d$-dimensional objects. Since homomorphisms of
finite-dimensional division algebras are injective, it follows that $\mathcal A_d$ is a (not necessarily small) groupoid for
each $d\in\{1,2,4,8\}$. For $d>1$, we have the following decomposition of $\mathcal A_d$, due to \cite{DD}.

\begin{Prp}\label{P: Double Sign} Let $A\in \mathcal A_d$ for some $d \in \{2,4,8\}$. For any $a,b \in A\setminus\{0\}$,
\[\begin{array}{lll}\sgn(\det L_a)=\sgn(\det L_b)&\text{and}&\sgn(\det R_a)=\sgn(\det R_b).\end{array}\]
The \emph{double sign} of $A$ is the pair $(\sgn(\det L_a),\sgn(\det R_a))$, $a \in A\setminus\{0\}$. Moreover, for each $d \in
\{2,4,8\}$,
\begin{equation}\label{Coproduct}
\mathcal{A}_d = \coprod_{(i,j)\in \Z_2^2} \mathcal A_d^{ij}
\end{equation}
where $\mathcal A_d^{ij}$ is the full subcategory of $\mathcal A_d$ formed by all algebras having double sign $((-1)^i,(-1)^j)$.
\end{Prp}

\begin{Rk}\label{R: Double Sign} If $A$ is unital, then $A$ has double sign $(+,+)$, and each orthogonal isotope
$A_{f,g}$ thus has double sign $(\det g,\det f)$.
\end{Rk}

For each $d\in\{2,4,8\}$ and each subcategory $\mathcal B$ of $\mathcal A_d$, we use the notation
\[\mathcal B^{ij}=\mathcal B \cap \mathcal A_d^{ij}.\]

We stress that the superscript belongs to $\Z_2^2$. While it is more customary to use pairs of the signs $+$
and $-$, hence writing $\mathcal B^{\alpha\beta}$ where $\alpha=(-1)^i$ and $\beta=(-1)^j$, this will be notationally inconvenient
for our purposes. Throughout this article, we will abuse notation by writing and treating the elements
of $\Z_2$ as the natural numbers $0$ and $1$ with operations modulo 2.

In the study we are undertaking, we will use and generalize the concept of a \emph{description} of a groupoid, due to Dieterich
\cite{D}. In this sense,
a  \emph{description} of a groupoid $\mathcal C$ is a quadruple $(G,M,\alpha,\mathcal F)$ where $G$ is a group, $M$ is a set,
$\alpha: G\times M\to M$ 
a group action, and $\mathcal F: \phantom{}_GM\to\mathcal C$ an equivalence of categories. The groupoid $_GM$ is defined by
having $M$ as object set, and for each $x,y\in M$, the morphism set
\[_GM(x,y)=\{(g,x,y)|g\cdot x=y\}.\]
We denote the morphism $(g,x,y)$ by $g$ if the domain and codomain are clear from context. By constructing a description of a
groupoid, one transfers the isomorphism problem for this groupoid to the normal form problem for the group action involved, and
classifying the groupoid then amounts to finding a transversal for the orbits of the group action. 

The groups appearing in the descriptions below are sometimes constructed as semi-direct products. We include the definition in
order to fix notation.

\begin{Def} Let $G$ and $H$ be groups and $\beta: H\times G\to G$ a group action such that for each $h\in H$, the map $\beta_h:
g\mapsto h\cdot g$
is an automorphism of $G$. The \emph{semi-direct product of $G$ with $H$ with respect to $\beta$} is the group $G\rtimes
H=G\rtimes_\beta H$
with underlying set $G\times H$,
and multiplication
\[(g,h)(g',h')=(g\beta_h(g'), hh').\]
\end{Def}

\subsection{Derivations}
A \emph{derivation} on an algebra $A$ over a field $k$ is a linear map $\delta: A\to A$ satisfying the Leibniz rule
\[\delta(xy)=\delta(x)y+x\delta(y)\]
for any $x,y\in A$. If $\delta$ and $\delta'$ are two derivations of $A$, then so is the combination
$r\delta+s\delta'$ for each $r,s\in k$ and the commutator $\delta\delta'-\delta'\delta$. The set
of all derivations of $A$ is a Lie algebra over $k$ under the commutator, denoted $\Der(A)$. The algebra $A$ becomes a
$\Der(A)$-module under the action $\delta\cdot a=\delta(a)$.

For each $d\in\{1,2,4,8\}$, we will write $\mathcal D_d$ for the full subcategory of $\mathcal A_d$ whose object class consists
of those algebras having a non-abelian derivation algebra. In dimension at most four, the following
result captures the situation completely. 

\begin{Prp} 
\begin{enumerate}[(i)]
\item $\mathcal D_1=\mathcal D_2=\emptyset$ and $\mathcal D_4$ is classified by $\left\{\HH_{K^j,K^i}|(i,j)\in \Z_2^2\right\}$,
where $K$ is the standard involution $x\mapsto \overline{x}$ in $\HH$.
\item For each $(i,j)\in\Z_2^2$, $L=\Der\left(\HH_{K^j,K^i}\right)$ is of type $\mathfrak{su}_2$, and
$\HH_{K^j,K^i}=\R1\oplus 1^\bot$ as an
$L$-module, both direct summands being irreducible.
\item The automorphism group of $\HH_{K^j,K^i}$ is
\[SO_3=SO\left(1^\bot\right)=\left\{\kappa_q|q\in\T=\mathbb S(\HH)\right\}.\]
\item For each $(i,j)\in\Z_2^2$, a description of $\mathcal D_4^{ij}$ is given by $\left(SO_3, \{*\}, \alpha, \mathcal
E^{ij}\right)$, where
$\alpha$ is the trivial action of
$SO_3$ on the singleton set $\{*\}$, and 
\[\mathcal E^{ij}: \phantom{}_{SO_3}\{*\}\to\mathcal D_4^{ij}\]
acts on objects by $\{*\}\mapsto \HH_{K^j,K^i}$, and on morphisms as identity.
\end{enumerate}
\end{Prp}

The first two items are due to \cite{EP}, the third is well-known, and the
last is a direct consequence of these. As above, we will henceforth identify $SO_3$ with $\left\{\kappa_q|q\in\mathbb
S(\HH)\right\}$, where $\kappa_q:\HH\to\HH$ is defined by $x\mapsto qx\overline{q}$.

\begin{Rk} In the above proposition we saw that each $\HH_{K^j,K^i}$ decomposes into irreducible submodules, in other words, is
\emph{completely reducible} as a module over its derivation algebra. This is true for any finite-dimensional division composition
algebra over any
field of characteristic not 2 or 3, as was observed in \cite{P}.
\end{Rk}

For the remainder of the paper we will be preoccupied with the full subcategory $\mathcal D=\mathcal D_8\subset\mathcal A_8$. We
define, for each partition $\Pi$ of the natural number 8, the full subcategory $\mathcal D_\Pi$ of $\mathcal D$, in which the
objects are all $A\in \mathcal D$ such that $\Pi$ the partition of 8 given by the dimensions of the irreducible components of $A$
as a $\Der(A)$-module. 

\begin{Ex} For the algebra $\OO$ we know that $\Der(\OO)$ is a Lie
algebra of type $\mathfrak {g}_2$, and that as a $\mathfrak{g}_2$-module, $\OO$ is the direct sum of a (trivial) one-dimensional
module, and an irreducible seven-dimensional module. Thus $\OO\in\mathcal D_{\{1,7\}}$. In fact, $\{\OO\}$ exhausts $\mathcal
D_{\{1,7\}}^{00}$ up to isomorphism.
\end{Ex}

In the sequel we shall omit the brackets in the subscript $\Pi$, except in the case where $\Pi=\{8\}$, to avoid confusion with
$\mathcal D_8=\mathcal D$

From \cite{P} we have the following decomposition of $\mathcal D$.

\begin{Prp}\label{P: Decomp} The category $\mathcal D$ decomposes as the coproduct
\[\mathcal D_{1,7}\amalg\mathcal D_{\{8\}}\amalg\mathcal D_{1,1,6}\amalg\mathcal D_{1,3,4}\amalg\mathcal D_{1,1,2,4}\amalg
\mathcal D_{1,1,1,1,4}\amalg\mathcal D_{3,5}\amalg\mathcal D_{1,1,3,3}.\]
For each $A\in\mathcal D$, $\Der(A)$ is of type $\mathfrak {g}_2$ if $A\in\mathcal D_{1,7}$, $\mathfrak {su}_3$ if $\A\in\mathcal
D_{\{8\}}\amalg \mathcal D_{1,1,6}$, $\mathfrak {su}_2\times\mathfrak {su}_2$ or $\mathfrak {su}_2\times\mathfrak {a}$ if $A\in
\mathcal D_{1,3,4}$, and $\mathfrak {su}_2\times\mathfrak {a}$ otherwise, where $\mathfrak a$ is an abelian Lie algebra with
$\dim\mathfrak a\leq 1$.
\end{Prp}

We will write $\mathcal D_{1,3,4}^s$ and $\mathcal D_{1,3,4}^a$ for the full subcategory of $\mathcal D_{1,3,4}$
in which the objects have a derivation algebra of type $\mathfrak {su}_2\times\mathfrak {su}_2$ and $\mathfrak
{su}_2\times\mathfrak {a}$, respectively. Thus $\mathcal D_{1,3,4}=\mathcal D_{1,3,4}^s\amalg\mathcal D_{1,3,4}^a$.

For each $A\in\mathcal D$, the sum of all one-dimensional $\Der(A)$-submodules of $A$ is the trivial submodule
\[A_0=\{a\in A|\forall \delta\in\Der(A): \delta(a)=0\}.\]
The following result is well-known.

\begin{Lma}\label{L: A0} Let $A\in \mathcal D$. Then $A_0$ is a proper composition subalgebra of $A$. Hence
$\dim A_0\in\{0,1,2,4\}$.
\end{Lma}

\begin{Rk}\label{R: Submodules}
Let $A,B\in\mathcal D$ and assume that $\phi:A\to B$ is an isomorphism. Then for each $\delta\in\Der(A)$ we have
$\phi\delta\phi^{-1}\in\Der(B)$. It follows that $\phi(A_0)=B_0$. Moreover,
if $C\subseteq A$ is an irreducible $\Der(A)$-submodule, then $\phi(C)$ is an irreducible $\Der(B)$-submodule. From the fact (see
\cite{EP}) that $\delta(x)\bot x$ for each $x\in A$, we see that the orthogonal projection of $C$ on any
submodule of $A$ is itself a submodule. From this it follows that if $A\in\mathcal D_\Pi$, and the natural number $d=\dim C$
occurs precisely once in the partition $\Pi$, then $C$ is the unique $d$-dimensional submodule of $A$, and $\phi(C)$ is the unique
$d$-dimensional submodule of $B$.
\end{Rk}

We will refer to the results on derivation algebras of division composition algebras obtained in \cite{P} along the way as we
need them. Those results which are not necessary for our arguments will be left out in order to keep this article within
reasonable limits. The reader is referred to \cite{P} for these results and their proofs.

\section{Eight-Dimensional Real Division Composition Algebras}
By Proposition \ref{P: Albert}, each $A\in \mathcal A_8$ is isomorphic to $\OO_{f,g}$ for some $f,g\in O(\OO)$. In this section,
we will recall some results on when two such algebras are isomorphic. To begin with, we will review some basic properties
of the automorphism group of $\OO$. This is a 14-dimensional compact
exceptional Lie group of type $G_2$, which we will denote by $G_2$ for notational simplicity. Identifying the groups $O_8$ and
$SO_8$ with $O(\OO)$ and $SO(\OO)$, respectively, we have $G_2<SO_8$. One way to visualize $G_2$ is by
using Cayley triples. A \emph{Cayley triple} is an orthonormal triple $(u,v,z)$, where $u,v,z\in\OO$ are purely imaginary (i.e.\
orthogonal to 1) and satisfy $z\bot uv$. Any permutation of a Cayley triple is a Cayley triple. Given a Cayley triple $(u,v,z)$,
the element $u$ generates the two-dimensional subalgebra $\langle 1,u\rangle \simeq\C$, and $\{u,v\}$ generates the
four-dimensional subalgebra
$\langle 1,u,v,uv\rangle \simeq\HH$.\footnote{If $V$ is a vectorspace and
$u_1,\ldots,u_n\in V$, then $\langle u_1,\ldots,u_n\rangle$ always denotes their linear span.} It is known that the group $G_2$
acts transitively on the set of all Cayley triples, and that
for any two
such triples there is a unique $\phi\in G_2$ mapping one to the other. Thus by fixing a Cayley triple $(u,v,z)$ once and for all,
the elements of $G_2$ correspond bijectively to the set of all Cayley triples via $\phi\mapsto(\phi(u),\phi(v),\phi(z))$. \\

\emph{For the remainder of this article, we fix a Cayley triple $(u,v,z)$ in $\OO$. In view of the above, we will often identify
the subalgebras $\langle 1,u\rangle$ and $\langle 1,u,v,uv\rangle$ of $\OO$ with $\C$ and $\HH$, respectively. Thus $\C$ and
$\HH$ become subspaces of $\OO_{f,g}$ for each $f,g\in O_8$.
Moreover, we denote $\mathbb S(\C)$ by $\mathbb S^1$ and $\mathbb S(\HH)$ by $\T$. Finally, we denote the standard involution
$x\mapsto\overline{x}$ on $\OO$ (and thus on $\C$ and $\HH$ as well) by $K$.}\\

Proposition \ref{P: Albert} and the above discussion imply the following well-known fact.

\begin{Lma}\label{L: Canonical} Let $A\in\mathcal A_8$ be unital, $B\subset A$ a four-dimensional unital subalgebra,
and $C\subset B$ a two-dimensional unital subalgebra. Then there is an isomorphism
$\phi:A\to\OO$ such that $\phi(B)=\langle 1,u,v,uv\rangle $ and $\phi(C)=\langle 1,u\rangle $.
\end{Lma}

We will often need to establish isomorphisms between  isotopes of arbitrary unital algebras in $\mathcal A_8$ and isotopes of
$\OO$. For this, we use the following result from
\cite{Ca}.

\begin{Lma}\label{L: Juggling} Let $A,B\in \mathcal A_8$ and assume there is an isomorphism $\phi: A\to B$. Then for each $f,g\in
O_8$, $\phi: A_{f,g}\to B_{\phi f\phi^{-1},\phi g \phi^{-1}}$
is an isomorphism.
\end{Lma}

This was used in \cite{Ca} to prove the following isomorphism condition.

\begin{Prp}\label{P: Ca} Let $f,g,f',g'\in O_8$. Then $\phi:\OO_{f,g}\to\OO_{f',g'}$ is an isomorphism if and
only if $\phi\in SO_8$ and 
\begin{equation}\label{Iso}
 \begin{array}{lll} \phi_1f\phi^{-1}=f' &\text{and}&\phi_2g\phi^{-1}=g',\end{array}
\end{equation}
where $(\phi_1,\phi_2)$ is a triality pair of $\phi$.
\end{Prp}

A \emph{triality pair} of $\phi\in SO_8$ is a pair $(\phi_1,\phi_2)\in SO_8^2$ satisfying $\phi(xy)=\phi_1(x)\phi_2(y)$ for any
$x,y\in\OO$. The Principle of triality, due to E.\ Cartan \cite{C}, asserts that for each $\phi\in SO_8$, a triality pair exists
and is
unique up to sign. It follows from the definition that if $(\phi_1,\phi_2)$ is a triality pair of $\phi$, then
\begin{equation}\label{E: Triality}
 \begin{array}{lll}\phi_1=R_{\overline{\phi_2(1)}}^\OO\phi &\text{and}& \phi_2=L_{\overline{\phi_1(1)}}^\OO\phi.\end{array}
\end{equation}

Using the above result, we obtained in \cite{A2} a description of the category $\mathcal A_8$. This description is given by the
quadruple 
\[(SO_8,[O_8\times O_8],\tau,\mathcal T),\]
where $[O_8\times O_8]$ is the quotient $(O_8\times O_8)/\{\pm(\Id,\Id)\}$, and $\tau$ is the \emph{triality action}
\[\begin{array}{ll}SO_8\times[O_8\times O_8]\to[O_8\times O_8],& (\phi,[f,g])\mapsto
\left[\phi_1 f\phi^{-1},\phi_2 g \phi^{-1}\right],\end{array}\]
where $(\phi_1,\phi_2)$ is any triality pair of
$\phi$. The functor 
\[\mathcal T: \phantom{}_{SO_8}[O_8\times O_8]\to \mathcal A_8\]
acts on objects by $[f,g]\mapsto\OO_{f,g}$, and on morphisms by $\phi\mapsto\phi$.

By definition of an automorphism, $\phi\in SO_8$ belongs to $G_2$ if and only if $(\phi,\phi)$ is a triality pair of
$\phi$. (By \eqref{E: Triality}, this is equivalent to $\phi$ having a triality pair $(\phi_1,\phi_2)$ with
$\phi(1)=\phi_1(1)=\phi_2(1)=1$.) Thus if one
restricts the triality action to $G_2<SO_8$, one obtains the action
\[\begin{array}{ll}
\gamma: G_2\times [O_8\times O_8]\to [O_8\times O_8], & (\phi,[f,g])\mapsto\left[\phi f\phi^{-1},\phi g \phi^{-1}\right].
  \end{array}
\]
However, the inclusion functor
\[\phantom{}_{G_2}[O_8\times O_8]\to \phantom{}_{SO_8}[O_8\times O_8],\]
is obviously not full, and therefore not all morphisms in $\mathcal A_8$ belong to $G_2$. The following lemma presents one
condition under which all morphisms between two algebras $\OO_{f,g}$ and $\OO_{f',g'}$ belong to $G_2$. As it turns out, it will
often be the case that the algebras we consider here satisfy this condition.

\begin{Lma}\label{L: G2} Let $f,g,f',g'\in O_8$. If there exists a subspace $\{0\}\neq U\subseteq \R^8$ such that the restriction
of
each of $f$, $g$, $f'$ and $g'$ to $U$ is the identity, then \linebreak $\phi:\OO_{f,g}\to\OO_{f',g'}$ is an isomorphism with
$\phi(U)=U$ if
and only if $\phi\in G_2$, $\phi(U)=U$ and $(f',g')=\left(\phi f\phi^{-1},\phi g \phi^{-1}\right)$.
\end{Lma}

\begin{proof} 
If $\phi\in G_2$ and $(f',g')=\left(\phi f\phi^{-1},\phi g \phi^{-1}\right)$, then $(f',g')=\left(\phi_1
f\phi^{-1},\phi_2g\phi^{-1}\right)$ for the triality pair $(\phi_1,\phi_2)=(\phi,\phi)$ of $\phi$, and by Proposition \ref{P: Ca},
$\phi$ is an
isomorphism. If, conversely, $\phi: \OO_{f,g}\to\OO_{f',g'}$ is an isomorphism, then for some triality pair $(\phi_1,\phi_2)$ of
$\phi$,
\[\begin{array}{lll} R^\OO_{\overline{\phi_2(1)}}\phi f \phi^{-1}=\phi_1 f\phi^{-1}=f'& \text{and}&
L_{\overline{\phi_1(1)}}^\OO\phi g \phi^{-1}=\phi_2 g\phi^{-1}=g'.\end{array}
                                                                                                                      \]
Now $\phi f\phi^{-1}$ and $\phi g\phi^{-1}$ act as identity on $U$, whence applying both sides of the two above equalities to any
$u\in U\setminus 0$ one obtains the equalities
\[\begin{array}{lll} u=\phi_2(1)u & \text{and}& u=\phi_1(1)u\end{array}\]
in $\OO$. Since $\OO$ is a unital division algebra we get $\phi_1(1)=\phi_2(1)=1$, whence\linebreak $\phi(1)=1$, and altogether
$\phi\in
G_2$. Thus
$[f',g']=\left[\phi f\phi^{-1},\phi g\phi^{-1}\right]$, and then $(f',g')=(\phi f\phi^{-1},\phi g\phi^{-1})$ since e.g.\
$\phi f\phi^{-1}$ and $f'$ must have the same
sign,
by the condition on their restrictions to $U$.
\end{proof}

\begin{Rk}\label{R: U1} Choosing $U=\langle 1\rangle$ in the lemma, we deduce the fact obtained in \cite{Ca} that if
$f(1)=g(1)=f'(1)=g'(1)=1$,
then $\phi:\OO_{f,g}\to\OO_{f',g'}$ is an isomorphism fixing 1 if and only if $\phi\in G_2$ and $(f',g')=(\phi f\phi^{-1},
\phi g \phi^{-1})$.
\end{Rk}

\section{Descriptions in Dimension Eight}
In this section, we will construct a description of the subcategory $\mathcal B^{ij}$ of $\mathcal D$ for each
$(i,j)\in\Z_2^2$, and each $\mathcal B$ among
\[\begin{array}{llllll}
\mathcal D_{1,7},&
\mathcal D_{1,1,6},&
\mathcal D_{1,3,4},&
\mathcal D_{1,1,2,4}&\text{and}&
\mathcal D_{1,1,1,1,4}.
\end{array}\]

The subcategory $\mathcal D_{\{8\}}$ has been studied in several papers, giving a classification which we quote
below. We will return to all these subcategories in Section 4, using the approach via quasi-descriptions. The reader may thus see
this section as a means of paving the road for the next. The remaining two subcategories $\mathcal D_{3,5}$ and $\mathcal
D_{1,1,3,3}$ will be treated in Section 5.

\subsection{The Subcategory $\mathcal D_{1,7}$}
Let $A\in\mathcal A_8$. By Proposition \ref{P: Decomp}, $A\in\mathcal D_{1,7}$ if and only if $\Der(A)$ is of type $\mathfrak
{g}_2$. By \cite{EP}, this holds if and
only if
$A$ is \emph{standard}, that is, 
$A\simeq\OO_{K^j,K^i}$ for some $(i,j)\in\Z_2^2$. For each of these isotopes, the trivial submodule is the subspace $\langle
1\rangle$. The fact that these four isotopes of $\OO$ are pairwise non-isomorphic is
well-known (and follows from e.g.\ Proposition \ref{P: Double Sign}), as is the fact that their
automorphism groups coincide with the automorphism group of $\OO$. From these observations, the following descriptions of the
double sign
components of $\mathcal D_{1,7}$ are immediately obtained.

\begin{Prp}\label{P: 1,7} The group $G_2$ acts trivially on the singleton set $\{*\}$, and for each $(i,j)\in\Z_2^2$, the map
\[\phantom{}_{G_2}\{*\}\to \mathcal D_{1,7}^{ij}\]
acting on objects by $\{*\}\mapsto \OO_{K^j,K^i}$, and on morphisms by $\phi\mapsto\phi$, is an
equivalence of categories. 
\end{Prp}

\subsection{The Subcategories $\mathcal D_{\{8\}}$ and $\mathcal D_{1,1,6}$}
We begin with the following characterization of $\mathcal D_{\{8\}}$ due to \cite{EP}.

\begin{Lma} Let $A\in\mathcal A_8$. Then $A\in \mathcal D_{\{8\}}$ if and only if $A$
is an Okubo algebra. 
\end{Lma}

Okubo algebras have been treated extensively in the literature. For a definition, see e.g\ \cite{EM}. The below construction is
readily derived from \cite{P}.

\begin{Ex}\label{E: EM}
Define the algebra
\[P^{11}=\OO_{K\tau,K\tau^{-1}},\]
where $\tau=\tau_{\left(\sqrt 3 u-1\right)/2}\in G_2$ is defined by $(u,v,z)\mapsto \left(u,v,z\left(\sqrt 3 u-1\right)/2\right)$.
Then $P^{11}$ is a division
Okubo algebra with double sign $(-,-)$.
\end{Ex}

In \cite{EM} it was shown that there is a unique isomorphism class of real division Okubo algebras. This implies that $\mathcal
D_{\{8\}}=\mathcal D_{\{8\}}^{11}$ is classified by $\{P^{11}\}$. 

For $\mathcal D_{1,1,6}$, a dense subset together with an isomorphism condition were given in \cite{EP} as
follows.

\begin{Lma}\label{L: EP} Let $A\in\mathcal A_8$. Then $A\in\mathcal D_{1,1,6}$ if and only if $A\simeq
C_{f,g}$, where $C\in\mathcal A_8$ is unital with standard involution $J$, and where for some two-dimensional unital subalgebra
$B$ of $C$, the pair $(f,g)\in O(C)^2$ satisfies
\begin{enumerate}[(i)]
\item $f(B)=g(B)=B$ and $f|_{B^\bot}=g|_{B^\bot}=\Id|_{B^\bot}$, and
\item $f, g\notin\{\Id,-J\}$.
\end{enumerate}
Moreover, if $(f,g),(f',g')\in O(C)^2$ satisfy (i) and (ii) for subalgebras $B$ and $B'$, respectively, then
$\phi: C_{f,g}\to C_{f',g'}$ is an isomorphism if and only if $\phi\in\Aut(C)$ with $\phi f \phi^{-1}=f'$ and $\phi
g\phi^{-1}=g'$. Finally if $B=B'$, then $C_{f,g}\simeq C_{f',g'}$ if and only if 
\[\begin{array}{lll} J^k f J^k=f'& \text{and} & J^lgJ^l=g'\end{array}\]
for some $(k,l)\in\Z_2^2$.
\end{Lma}

It was shown in \cite{EP} that if $A=C_{f,g}$ is as in the lemma, then $A_0=B$.

Using this, we next give a description of $\mathcal D_{1,1,6}^{ij}$ for each $(i,j)\in\Z_2^2$. For
this, we
consider the subgroup
\[G_2^{\langle u\rangle }=\{\phi\in G_2|\phi(u)=\pm u\}\]
of $G_2$. It is known that the subgroup of all $\phi\in G_2$ fixing $u$ is isomorphic to $SU_3$. Identifying $SU_3$ with
this subgroup, we see that $G_2^{\langle u\rangle }$ is isomorphic to the semi-direct product $SU_3\rtimes\Z_2$ with respect to
the
action 
\[\begin{array}{ll} \Z_2\times SU_3\to SU_3 & (\varepsilon, \rho)\mapsto
\widehat{\varepsilon}\rho\widehat{\varepsilon}\end{array},\]
where $\widehat{\varepsilon}\in G_2$ is defined by $(u,v,z)\mapsto((-1)^\varepsilon u,v,z)$. The isomorphism
takes each $(\rho,\varepsilon)\in SU_3\rtimes \Z_2$ to
$\rho\widehat\varepsilon\in G_2^{\langle u\rangle }$. The description is given by the following result.

\begin{Prp}\label{P: 1,1,6} Let $(i,j)\in\Z_2^2$. The group $SU_3\rtimes\Z_2$ acts on 
\[(\mathbb S^1\times \mathbb S^1)_{ij}=(\mathbb S^1\times \mathbb S^1)\setminus\{((-1)^j,(-1)^i)\}\]
by
\begin{equation}\label{E: Action0}
(\rho,\varepsilon)\cdot (a,b)=(K^\varepsilon(a),K^\varepsilon(b)).
\end{equation}
The map
\[\mathcal F^{ij}: \phantom{}_{SU_3\rtimes\Z_2}(\mathbb S^1\times \mathbb S^1)_{ij}\to \mathcal D_{1,1,6}^{ij}\]
acting on objects by $(a,b)\mapsto\OO_{\lambda_a^{(j)},\lambda_b^{(i)}}$, and on morphisms by
$(\rho,\varepsilon)\mapsto\rho\widehat\varepsilon$, is
an equivalence of categories.
\end{Prp}

For each $t\in\mathbb S^1$ and $k\in\Z_2$, the map
$\lambda_t^{(k)}:\OO\to\OO$
is defined by

\[   \lambda_t^{(k)}(x)=\left\{ \begin{array}{ll} tK^k(x) & \text{if\ } x\in\C,\\
                                           x  & \text{if\ } x\in\C^\bot.
  \end{array}\right.\]
Thus with $C=\OO$, $B=\C$ and $(f,g)=(\lambda_a^{(j)},\lambda_b^{(i)})$, the algebra $C_{f,g}$ satisfies items (i) and (ii) of
Lemma
\ref{L: EP}.

\begin{proof} Let $(i,j)\in\Z_2^2$. Equation \eqref{E: Action0} defines an action of $SU_3\rtimes\Z_2$ on $\mathbb S^1\times
\mathbb S^1$ which fixes $\{((-1)^j,(-1)^i)\}$ and thus induces the desired action on its complement. The map $\mathcal F^{ij}$
is well-defined on objects by Lemma \ref{L: EP}, since $\lambda_t^{(k)}\in\{\Id,-K\}$ if and only if $t=(-1)^k$. To
show that it is well-defined on morphisms, we first note that it maps identity morphisms to identity morphisms. Further, assume
that $(\rho,\varepsilon)\cdot (a,b)=(c,d)$ for some $(\rho,\varepsilon)\in
SU_3\rtimes \Z_2$ and $(a,b),(c,d)\in \mathbb S^1\times \mathbb S^1$. Then for each $x\in\C$,
\begin{equation}\label{E: t}
\rho\widehat\varepsilon \lambda_a^{(j)}
(\rho\widehat\varepsilon)^{-1}(x)=K^\varepsilon(aK^jK^\varepsilon(x))=K^\varepsilon(a)K^j(x)=\lambda_{K^\varepsilon(a)}^{(j)}(x),
\end{equation}
and on $\C^\bot$ both $\rho\widehat\varepsilon \lambda_a^{(j)} (\rho\widehat\varepsilon)^{-1}$ and $\lambda_c^{(j)}$ act as
identity. Thus
$\rho\widehat\varepsilon \lambda_a^{(j)}
(\rho\widehat\varepsilon)^{-1}=\lambda_c^{(j)}$ and, analogously, $\rho\widehat\varepsilon
\lambda_b^{(i)}(\rho\widehat\varepsilon)^{-1}=\lambda_d^{(i)}$. Since
$\rho\widehat\varepsilon\in
G_2$, Lemma \ref{L: EP} implies that 
\[\rho\widehat\varepsilon: \OO_{\lambda_a^{(j)},\lambda_b^{(i)}} \to\OO_{\lambda_c^{(j)},\lambda_d^{(i)}},\]
whence $\mathcal F^{ij}$ is well-defined on morphisms. Functoriality (i.e.\ the property that the map respects composition of
morphisms) is due to the fact that
$(\rho,\varepsilon)\mapsto\rho\widehat\varepsilon$ defines a group homomorphism from $SU_3\rtimes\Z_2$ to $G_2^{\langle
u\rangle}$.

To prove denseness, let $A\in\mathcal D_{1,1,6}^{ij}$. Then Lemma \ref{L: EP} implies that
$A\simeq
C_{f,g}$ for a unital algebra $C\in\mathcal A_8$ and $f,g\in O(C)$ that map a two-dimensional unital subalgebra $B\subset C$ to
itself
and are the identity on $B^\bot$. Lemma \ref{L: Canonical} then implies that there is an isomorphism $\phi:C\to\OO$ mapping $B$ to
$\C$. Thus by Lemma \ref{L: Juggling} we have $A\simeq \OO_{f',g'}$ with $f'=\phi f\phi^{-1}$ and $g'=\phi g\phi^{-1}$ mapping
$\C$ to itself and being the identity on $\C^\bot$. By double sign considerations, $\det f'=j$ and $\det g'=i$. Recall that an
orthogonal map on $\C$ with determinant $k$ is of the form $x\mapsto tK^k(x)$ for some $t\in\mathbb S^1$. Hence there exist
$a,b\in \mathbb S^1$ such that $f'=\lambda_a^{(j)}$ and $g'=\lambda_b^{(i)}$. Finally condition (ii) from Lemma \ref{L: EP}
ensures that
$(a,b)\neq((-1)^j,(-1)^i)$. Thus $A$
is isomorphic to $\mathcal F^{ij}(a,b)$, proving denseness.

The functor is faithful since if $\rho\widehat\varepsilon=\sigma\widehat\eta$, then applying both sides to the Cayley triple
$(u,v,z)$ we have $((-1)^\varepsilon u,\rho(v),\rho(z))=((-1)^\eta u,\sigma(v),\sigma(z))$. Thus clearly $\varepsilon=\eta$, and
$\rho=\sigma$ as both maps fix $u$. To prove fullness, assume that   
\[\phi: A=\OO_{\lambda_a^{(j)},\lambda_b^{(i)}} \to\OO_{\lambda_c^{(j)},\lambda_d^{(i)}}=B.\]
By Remark \ref{R: Submodules}, $\phi$ maps the trivial $\Der(A)$-submodule $\C\subset A$ to the trivial $\Der(B)$-submodule
$\C\subset B$, whence
$\phi\left(\C^\bot\right)=\C^\bot$, and Lemma \ref{L: G2} applies, giving that $\phi\in G_2^{\langle u\rangle }$ and satisfies
\[\begin{array}{lll}\phi \lambda_a^{(j)} \phi^{-1}=\lambda_c^{(j)} & \text{and} &\phi \lambda_b^{(i)}
\phi^{-1}=\lambda_d^{(i)}.\end{array}\]
The isomorphism $SU_3\rtimes \Z_2\to G_2^{\langle u\rangle}$ established above implies that that $\phi=\rho\widehat\varepsilon$
for some $(\rho,\varepsilon)\in SU_3\rtimes
\Z_2$, and from \eqref{E: t} we conclude that $\lambda_{K^\varepsilon(a)}^{(j)}=\lambda_c^{(j)}$. Applying both sides to $1$ gives
$K^\varepsilon(a)=c$, and by analogy we get $K^\varepsilon(b)=d$. Altogether $(\rho,\varepsilon)\cdot (a,b)=(c,d)$ and the functor
is full, which completes the proof.
\end{proof}

\subsection{The Subcategory $\mathcal D_{1,3,4}$}

Following \cite{P} we
define, for each $p\in \T\subset\HH\subset\OO$, the map $\tau_p\in G_2$ by
\[\tau_p: (u,v,z)\mapsto (u,v,zp).\]
Moreover we define, for each $q\in \T$, the map $\widehat{\kappa}_q\in G_2$ by
\[\widehat{\kappa}_q: (u,v,z)\mapsto (qu\overline{q},qv\overline{q},z),\]
which is well-defined as $\HH$ is an associative subalgebra of $\OO$. Recall that we have defined $\kappa_q:\HH\to\HH, x\mapsto
qx\overline{q}$ for each $q\in\T$, and $\kappa_q$ is the restriction of $\widehat{\kappa}_q$ to $\HH\subset\OO$.
Observe
however that $\widehat\kappa_q$ does not act as identity on $\HH^\bot=\HH z$; indeed $\widehat{\kappa}_q(xz)=\kappa_q(x)z$ for
each
$x\in\HH$.

These maps are useful to characterize those automorphisms of $\OO$ which fix $\HH$, either pointwise or as a subalgebra. Thus we
introduce the notation
\[\begin{array}{lll}G_2^\HH=\{\phi\in G_2|\phi(\HH)=\HH\} &\text{and}& G_2^{(\HH)}=\{\phi\in
G_2^\HH|\phi|_\HH=\Id|_\HH\}.\end{array}\]
It is well-known (see e.g.\ \cite{Sa}) that $G_2^\HH\simeq SO_4$, while it was shown in \cite{P} that
$G_2^{(\HH)}=\{\tau_p|p\in\T\}$. The following lemma encompasses these results and is adapted to suit our current
formalism. Recall that we identify $SO_3$ with $\{\kappa_q|q\in\T\}$, and consider the semi-direct product $\T\rtimes SO_3$
with respect to the action of $SO_3$ on $\T$ given by $\kappa_q\cdot                                
p=\kappa_q(p)$. 

\begin{Lma}\label{L: Tau} 
The map
\[\Delta: \begin{array}{ll}\mathbb S^3\rtimes SO_3\to G_2^\HH &(p,\kappa_q)\mapsto\tau_{\overline{p}}\widehat\kappa_q\end{array}\]
is an isomorphism of groups, inducing an isomorphism $\T\to G_2^{(\HH)}$, $p\mapsto\tau_{\overline{p}}$.
\end{Lma}

\begin{proof} To prove that $\Delta$ is a group homomorphism we must show that for any $p,q,p',q'\in\T$,
$\tau_{\overline{p}}\widehat\kappa_q\tau_{\overline{p'}}\widehat\kappa_q'$ equals
\[\Delta((p,\kappa_q)(p',\kappa_{q'}))=\tau_{\overline{pqp'\overline{q}}}\widehat\kappa_{qq'}.\]
Indeed, both map each $x\in\HH$ to $\kappa_{qq'}(x)$, while
\[\tau_{\overline{p}}\widehat\kappa_q\tau_{\overline{p'}}\widehat\kappa_q'(z)=\tau_{\overline{p}}\widehat\kappa_q(z\overline{p'}
)=\tau_{
\overline{p}}(z(q \overline{p'}\overline{q}))=
(z\overline{p})(q\overline{p'}\overline{q})=z(\overline{pqp'\overline{q}})=\tau_{\overline{pqp'\overline{q}}}\widehat\kappa_{qq'}
(z),
\] 
where we have used that 
\begin{equation}\label{E: Trick}
\forall x,y\in\HH, (zx)y=z(yx),
\end{equation}
which can be verified from a multiplication table of $\OO$. 

The map is injective since if $\tau_{\overline{p}}\widehat{\kappa}_q=\tau_{\overline{p'}}\widehat{\kappa}_{q'}$, then restricting
to
$\HH$ we get $\kappa_q=\kappa_{q'}$, while $z{\overline{p}}=z{\overline{p'}}$ implies that $p=p'$. To prove surjectivity, assume
that $\rho\in G_2$ satisfies $\rho(\HH)=\HH$. Then $\rho|_\HH$ is an automorphism of the algebra $\HH$, and $\rho(z)\bot\HH$. Then
there
exists  $q\in \T$ and $z'\in\mathbb S\left(\HH^\bot\right)$ such that $\rho$ maps $(u,v,z)$ to
$(qu\overline{q},qv\overline{q},z')$. Then $z'=z\overline{p}$ with $p=\overline{z'}z\in \HH$, and
$\rho=\Delta(p,\kappa_q)$. The statement on the induced map holds since $\tau_{\overline{p}}\widehat\kappa_q$ fixes $\HH$
pointwise if and only if $\kappa_q=\Id|_\HH$, i.e.\ if and only if
$\tau_{\overline{p}}\widehat\kappa_q\in\Delta(\T\rtimes\{1\})$.
\end{proof}

\begin{Rk} The above lemma and remarks imply the fact that $\T\rtimes SO_3\simeq SO_4$. This isomorphism can also be
established directly.
Identifying $SO_4$ with $SO(\HH)$, it is a classical result that each $\phi\in SO_4$ is of the form $x\mapsto axb$ for some
$a,b\in \T$. One can then check that an isomorphism $\T\rtimes SO_3\to SO_4$ is obtained by mapping $(p,\kappa_q)$ to the map
$x\mapsto p\kappa_q(x)$.
\end{Rk}

Consider now the category $\mathcal D_{1,3,4}$, an exhaustive list of which is given in \cite{P}. This is refined in the following
result. 

\begin{Lma}\label{L: Dense1} Let $(i,j)\in\Z_2^2$ and $A\in\mathcal D_{1,3,4}^{ij}$. Then $A\simeq\OO_{K^j\tau_a,K^i\tau_b}$ for
some $(a,b)\in\left(\T\times\T\right)_{ij}$. Conversely,
$\OO_{K^j\tau_a,K^i\tau_b}\in\mathcal D_{1,3,4}^{ij}$ for each $(i,j)\in\Z_2^2$ and $(a,b)\in\left(\T\times\T\right)_{ij}$.
\end{Lma}

Here, 
\[\left(\T\times\T\right)_{ij}=\left\{ \begin{array}{ll} \left(\T\times\T\right)\setminus\{(a,a^2)|a^2+a+1=0\vee a=1\} & \text{if\
}(i,j)=(1,1),\\
                                           \left(\T\times\T\right)\setminus\{(1,1)\}  & \text{otherwise.}
  \end{array}\right.\]
 
\begin{proof} Let $A\in\mathcal D_{1,3,4}^{ij}$. By \cite{P} (Proposition 38), there exists $C\in\mathcal A_8$ with unity
$e$, a fixed quaternion subalgebra $Q$ and a fixed $w\in \mathbb S\left(Q^\bot\right)$ such that \[A\simeq
C_{J^l\tau'_x,J^k\tau'_y},\]
for some $(k,l)\in\Z_2^2$ and $x,y\in \mathbb S(Q)$, where $J$ is the standard involution on $C$ and $\tau'_x$ (resp.\ $\tau'_y)$
is the unique automorphism of $C$ fixing $Q$ pointwise and mapping $w$ to $wx$ (resp.\ $wy$). Then by Proposition \ref{P: Double
Sign} we necessarily have $(k,l)=(i,j)$. By Lemma \ref{L: Canonical} there is an
isomorphism 
\[\phi: C \to \OO\]
mapping $e$ to 1, $Q$ to $\HH$, and $w$ to $z$. Lemma \ref{L: Juggling} then implies that we have the isomorphism
\[\phi: C_{J^j\tau'_x,J^i\tau'_y} \to \OO_{\phi J^j\tau'_x\phi^{-1},\phi J^i\tau'_y\phi^{-1}}.\]
Since $\phi(e)=1$ we have $\phi J\phi^{-1}=K$. Moreover,
$\phi\tau'_x\phi^{-1}=\tau_{\phi(x)}$: indeed, $\tau_{\phi(x)}$ fixes $\HH$ pointwise,
and so does $\phi\tau'_x\phi^{-1}$ since $\phi^{-1}$ maps $\HH$ to $Q$. Moreover,
\[\phi\tau'_x\phi^{-1}(z)=\phi\tau'_x(w)=\phi(wx)=\phi(w)\phi(x)=z\phi(x)=\tau_{\phi(x)}(z).\]
Likewise, $\phi\tau'_y\phi^{-1}=\tau_{\phi(y)}$, whence $A\simeq\OO_{K^j\tau_a,K^i\tau_b}$ for some $(a,b)\in\T\times\T$.
Moreover, Proposition 19 in \cite{P} states that $C_{J^j\tau'_x,J^i\tau'_y}\in\mathcal D_{1,3,4}$ if and only if, firstly,
$(x,y)\neq (1,1)$ and, secondly, $(i,j)=(1,1)$ implies
either $x^2+x+e\neq 0$ or $y\neq x^2$ in $C$. This proves that $(a,b)\in\left(\T\times\T\right)_{ij}$ in all cases. The converse
holds by
applying Proposition 19 from \cite{P} with $C=\OO$ and $(x,y)=(a,b)$.
\end{proof}

In our construction of a description, we will let the group $\T\rtimes SO_3$ from Lemma \ref{L: Tau} act on
the set $\left(\T\times\T\right)_{ij}$ for each $(i,j)\in \Z_2^2$. The details are as follows.

\begin{Prp}\label{P: 1,3,4} Let $(i,j)\in \Z_2^2$. The group $\T\rtimes SO_3$ acts on $\left(\T\times\T\right)_{ij}$ by
\begin{equation}\label{E: Action}
(p,\kappa_q)\cdot (a,b)=(pqa\overline{pq}, pqb\overline{pq}). 
\end{equation}
The map
\[\mathcal G^{ij}: \phantom{}_{\T\rtimes SO_3}\left(\T\times\T\right)_{ij}\to \mathcal D_{1,3,4}^{ij}\]
acting on objects by
\[(a,b)\mapsto \OO_{K^j\tau_a,K^i\tau_b}\]
and on morphisms by $(p,\kappa_q)\mapsto\Delta((p,\kappa_q))$, is an equivalence of categories.
\end{Prp}

The isomorphism $\Delta:\T\rtimes SO_3\to G_2^\HH$ was defined in Lemma \ref{L: Tau} by \linebreak $(p,\kappa_q)\mapsto
\tau_{\overline{p}}\widehat{\kappa}_q$.

\begin{proof} By definition of the semi-direct product, \eqref{E: Action} defines an action $\alpha$ of $\T\rtimes SO_3$
on $\T\times\T$. The point $(1,1)$ is a fixed point of $\alpha$, while for each $p,q\in\T$,
\[(p,\kappa_q)\cdot (a,a^2)=(pqa\overline{pq}, pqa^2\overline{pq})=(pqa\overline{pq}, (pqa\overline{pq})^2),\]
with
\[(pqa\overline{pq})^2+pqa\overline{pq}+1=pq(a^2+a+1)\overline{pq}.\]
Thus $\alpha$ induces an action of $\T\rtimes SO_3$ on $\left(\T\times\T\right)_{ij}$ for each $(i,j)\in\Z_2^2$.

Fix now $(i,j)\in\Z_2^2$. By Lemma \ref{L: Dense1}, $\mathcal G^{ij}$ is well-defined on objects. Next we show well-definedness on
morphisms. First, identities are mapped to identities. Using \eqref{E: Trick} and the fact that $\widehat{\kappa}_q,\tau_p\in G_2$
for each $p,q\in\T$, one obtains, for each
$w\in\T$, that
\begin{equation}\label{E: tau}
\tau_{\overline{p}}\widehat{\kappa}_q\tau_w(\tau_{\overline{p}}\widehat{\kappa}_q)^{-1}=\tau_{pqw\overline{pq}}.
\end{equation}
Thus if $(p,\kappa_q)\cdot
(a,b)=(c, d)$, then $\phi=\tau_{\overline{p}}\widehat{\kappa}_q\in G_2$ satisfies
\[\begin{array}{lll} \phi K^j\tau_a\phi^{-1}=K^j\tau_c &\text{and}&\phi K^i\tau_b\phi^{-1}=K^i\tau_d,\end{array}.\]
since $\phi$ commutes with $K$. By Remark \ref{R: U1}, this proves that
$\mathcal G^{ij}$ maps morphisms to morphisms. Functoriality follows from the fact that $\Delta$ is a group
homomorphism, and faithfulness from the fact that $\Delta$ is injective.

To show that $G^{ij}$ is full, let 
\[\phi:A=\OO_{K^j\tau_a,K^i\tau_b}\to B=\OO_{K^j\tau_c,K^i\tau_d}\]
with $(a,b),(c,d)\in \left(\T\times\T\right)_{ij}$. From \cite{P} we know that $A_0=B_0=\langle 1\rangle$, and that $\HH$ is a
submodule of both $A$ and $B$. From Remark \ref{R: Submodules} it therefore follows that $\phi(1)=\pm1$. By Remark \ref{R: U1} we
have $\phi\in G_2$ and $(\phi\tau_a\phi^{-1},\phi\tau_b\phi^{-1})=(\tau_c,\tau_d)$. Remark \ref{R: Submodules} also implies that
$\phi(\HH)=\HH$. Lemma
\ref{L: Tau} then gives that $\phi=\Delta(p,\kappa_q)$ for some $p,q\in\T$, whence by \eqref{E: tau}, 
\[\begin{array}{lll}\tau_{pqa\overline{pq}}=\tau_c&\text{and}&\tau_{pqb\overline{pq}}=\tau_d,\end{array}\]
which by Lemma \ref{E: tau} implies that $c=pqa\overline{pq}$ and $d=pqb\overline{pq}$. Thus $\phi=\mathcal G^{ij}(p,\kappa_q)$
with $(p,\kappa_q)\cdot (a,b)=(c, d)$, and fullness is proved. Finally, $\mathcal G^{ij}$ is dense by Lemma \ref{L: Dense1},
and the proof is complete.
\end{proof}

\begin{Rk} Let $(i,j)\in\Z_2^2$. Since $\mathcal D_{1,3,4}$ is the coproduct of the two subcategories $\mathcal
D_{1,3,4}^s$ and $\mathcal D_{1,3,4}^a$, it may be
called for to specify for which $(a,b)\in(\T\times\T)_{ij}$ the image $\mathcal G^{ij}(a,b)$ belongs to either one. This
information can be
deduced immediately from the above cited Proposition 19 in \cite{P}. Namely, $\mathcal G^{ij}(a,b)\in\mathcal D_{1,3,4}^s$ if
$\{a,b\}\subseteq\{\pm1\}$, and $\mathcal G^{ij}(a,b)\in\mathcal D_{1,3,4}^a$ otherwise. 
\end{Rk}

\subsection{The Subcategories $\mathcal D_{1,1,2,4}$ and $\mathcal D_{1,1,1,1,4}$}

We define, for each $a,b\in\HH\subset\OO$ and each $k\in \Z_2$, the map $T_{a,b}^{(k)}:\OO\to\OO$ by
\[   T_{a,b}^{(k)}(x)=\left\{ \begin{array}{ll} aK^k(x)b & \text{if\ } x\in\HH,\\
                                           x  & \text{if\ } x\in\HH^\bot.
  \end{array}\right.\]
For each $(i,j)\in\Z_2^2$, the next lemma gives a set exhausting $\mathcal D_{1,1,2,4}^{ij}\amalg\mathcal D_{1,1,1,1,4}^{ij}$.

\begin{Lma}\label{L: Dense2} 
Let $(i,j)\in\Z_2^2$ and $A\in\mathcal D_{1,1,2,4}^{ij}\amalg\mathcal D_{1,1,1,1,4}^{ij}$. Then $A$ is isomorphic to
$\OO_{T_{a_1,b_1}^{(j)},T_{a_2,b_2}^{(i)}}$ for some
$(a_1,b_1,a_2,b_2)\in S_{ij}$.
Conversely,
\[\OO_{T_{a_1,b_1}^{(j)},T_{a_2,b_2}^{(i)}}\in\mathcal D_{1,1,2,4}^{ij}\amalg\mathcal D_{1,1,1,1,4}^{ij}\]
for each $(i,j)\in\Z_2^2$ and $(a_1,b_1,a_2,b_2)\in S_{ij}$.
\end{Lma}

For each $(i,j)\in\Z_2^2$, $S_{ij}$ is defined as the set of all
$(a_1,b_1,a_2,b_2)\in\left(\T\right)^4\setminus\{\pm1\}^4$ which do
\emph{not} satisfy
\[\dim\langle \Im(a_k),\Im(b_k)|k\in\{1,2\}\rangle =1 \wedge a_1=(-1)^j b_1\wedge a_2=(-1)^i b_2,\]
where $\Im:\OO\to\OO$ is the orthogonal projection onto $1^\bot$.

The proof is similar in spirit to that of Lemma \ref{L: Dense1}; the details are as follows.
 
\begin{proof} Let $A\in\mathcal D_{1,1,2,4}^{ij}\amalg\mathcal D_{1,1,1,1,4}^{ij}$. By \cite{P} (Proposition 40), there exists
$C\in\mathcal A_8$ with unity
$e$, a
fixed quaternion
subalgebra $Q$, and elements $x_1,x_2,y_1,y_2\in Q$ with $\|x_iy_i\|=1$ for each $i\in\{1,2\}$, such that
\[A\simeq C_{T'^{(l)}_{x_1,y_1},T'^{(k)}_{x_2,y_2}}\]
for some $(k,l)\in\Z_2^2$. The map $T'^{(m)}_{x,y}:C\to C$ is defined for each $m\in\Z_2$ and each $x,y\in \mathbb S(Q)$ by
\[   T'^{(m)}_{x,y}(w)=\left\{ \begin{array}{ll} xJ^m(w)y & \text{if\ } w\in Q,\\
                                           w  & \text{if\ } w\in Q^\bot,
  \end{array}\right.\]              
where $J$ is the standard involution on $C$. Proposition
\ref{P: Double Sign} then gives $(k,l)=(i,j)$. As before, Lemma \ref{L: Canonical} guarantees the existence of an
isomorphism
\[\phi: C \to \OO\]
mapping $e$ to 1 and $Q$ to $\HH$. Thence by Lemma \ref{L: Juggling} we have the isomorphism
\[\phi: C_{T'^{(j)}_{x_1,y_1},T'^{(i)}_{x_2,y_2}} \to \OO_{\phi T'^{(j)}_{x_1,y_1}\phi^{-1},\phi T'^{(i)}_{x_2,y_2}\phi^{-1}}.\]
Note that upon rescaling $x$ and $y$ we may assume that $\|x_i\|=\|y_i\|=1, i\in\{1,2\}$. Next we prove that $\phi
T'^{(m)}_{x,y}\phi^{-1}=T^{(m)}_{\phi(x),\phi(y)}$ for each $m\in\Z_2$ and $x,y\in Q$ with $\|x\|=\|y\|=1$. Since $\phi$ maps $Q$
to
$\HH$, and therefore $Q^\bot$ to $\HH^\bot$, the left hand side fixes $\HH^\bot$ pointwise, and as does the right hand side. Given
any $w\in\HH$ we obtain, using the fact that $\phi:C\to\OO$ is an isomorphism, that
\[\phi T'^{(m)}_{x,y}\phi^{-1}(w)=\phi(xJ^m\phi^{-1}(w)y)=\phi(x)K^m(w)\phi(y)=T^{(m)}_{\phi(x),\phi(y)}(w).\]
Thus $A\simeq\OO_{T_{a_1,b_1}^{(j)},T_{a_2,b_2}^{(i)}}$ for some $(a_1,b_1,a_2,b_2)\in \left(\T\right)^4$. Now Proposition 28 in
\cite{P}
implies that 
$\OO_{T_{a_1,b_1}^{(j)},T_{a_2,b_2}^{(i)}}\in\mathcal D_{1,1,2,4}^{ij}\amalg\mathcal D_{1,1,1,1,4}^{ij}$ precisely when
$(a_1,b_1,a_2,b_2)\in S_{ij}$, which completes the argument. The
converse holds
by applying Proposition 28 from \cite{P} with $C=\OO$ and $(x_1,y_1,x_2,y_2)=(a_1,b_1,a_2,b_2)$.
\end{proof}

To obtain a description, we will again make use of the group $\T\rtimes SO_3$, acting on a set that we
will now construct. Consider thus the normal subgroup $\{\pm(1,1)\}$ of $\T\times\T$, and denote the corresponding
quotient group by $\left[\T\times\T\right]$, and its elements by $[a,b]$ with $a,b\in\T$. We obtain a surjective map
\[\left(\T\right)^4\to \left[\T\times\T\right]\times\left[\T\times\T\right], (a_1,b_1,a_2,b_2)\mapsto ([a_1,b_1],[a_2,b_2]),\]
and denote, for each $(i,j)\in\Z_2^2$, the image of $S_{ij}$ under this map by $[S_{ij}]$.

\begin{Prp}\label{P: 1114} Let $i,j\in\Z_2$. The group $\T\rtimes SO_3$ acts on $[S_{ij}]$ by
\begin{equation}\label{E: Action2}
(p,\kappa_q)\cdot ([a_1,b_1],[a_2,b_2])=([\kappa_q(a_1), \kappa_q(b_1)],[\kappa_q(a_2),\kappa_q(b_2)]). 
\end{equation}
The map
\[\mathcal I^{ij}: \phantom{}_{\T\rtimes SO_3}[S_{ij}]\to \mathcal D_{1,1,2,4}^{ij}\amalg\mathcal D_{1,1,1,1,4}^{ij}\]
acting on objects by
\[([a_1,b_1],[a_2,b_2])\mapsto \OO_{T_{a_1,b_1}^{(j)},T_{a_2,b_2}^{(i)}}\]
and on morphisms by $(p,\kappa_q)\mapsto \Delta((p,\kappa_q))$, is an equivalence of categories.
\end{Prp}

The map $\Delta$ was defined in Lemma \ref{L: Tau}.

\begin{proof} Since $\kappa_q$ is an automorphism of $\HH$, \eqref{E: Action2} defines an action of $\T\rtimes
SO_3$ on $\left[\T\times\T\right]\times\left[\T\times\T\right]$, mapping $[S_{ij}]$ to itself for each
$(i,j)\in\Z_2^2$.

Let now $(i,j)\in\Z_2^2$. The map $\mathcal I^{ij}$ is well-defined on objects by Lemma \ref{L: Dense2} and the fact that for
each $k\in \Z_2$ and each $a,b\in\T$,
\[T^{(k)}_{a,b}=T^{(k)}_{-a,-b}.\]
To show well-definedness on morphisms we first note that identities are mapped to identities. Moreover, for each
$a,b,p,q\in\mathbb
S(\HH)$ and
each $k\in\Z_2$,
\begin{equation}\label{E: T}
\tau_{\overline{p}}\widehat{\kappa}_qT_{a,b}^{(k)}(\tau_{\overline{p}}\widehat{\kappa}_q)^{-1}=T_{qa\overline{q},qb\overline{q}}^{
(k)}.
\end{equation}
Thus if
\[(p,\kappa_q)\cdot ([a_1,b_1],[a_2,b_2])=([c_1, d_1],[c_2,d_2])\]
then for $\phi=\tau_{\overline{p}}\widehat{\kappa}_q\in G_2$ we have
\[\begin{array}{lll} \phi T_{a_1,b_1}^{(j)}\phi^{-1}=T_{c_1,d_1}^{(j)} &\text{and}&\phi
T_{a_2,b_2}^{(i)}\phi^{-1}=T_{c_2,d_2}^{(i)}.\end{array}\]
Since $T^{(k)}_{a,b}$ acts as identity on $\HH^\bot$ for each $k\in\Z_2$ and $a,b\in\T$, and since $\phi$ maps $\HH^\bot$ to
itself, we may apply Lemma \ref{L: G2} to deduce that $\phi$ is indeed a morphism from
$\OO_{T_{a_1,b_1}^{(j)},T_{a_2,b_2}^{(i)}}$ to $\OO_{T_{c_1,d_1}^{(j)},T_{c_2,d_2}^{(i)}}$. Functoriality and faithfulness follow
from $\Delta$
being a group monomorphism, and denseness from Lemma \ref{L: Dense2}. To see that the functor is full, let
$A=\OO_{T_{a_1,b_1}^{(j)},T_{a_2,b_2}^{(i)}}$ and
$B=\OO_{T_{c_1d_1}^{(j)},T_{c_2d_2}^{(i)}}$ for some $(a_1,b_1,a_2,b_2)$ and $(c_1,d_1,c_2,d_2)$ in $S_{ij}$. From \cite{P} it
follows that $\HH^\bot$ is an irreducible submodule of dimension four of $A$ as well as of $B$. Thus by Remark \ref{R:
Submodules}, any isomorphism
$\phi: A\to B$ maps $\HH^\bot$ to
itself. We may then apply Lemma \ref{L: G2} to conclude that $\phi\in G_2$ and that 
\[\begin{array}{lll} \phi T_{a_1,b_1}^{(j)}\phi^{-1}=T_{c_1,d_1}^{(j)} &\text{and}&\phi
T_{a_2,b_2}^{(i)}\phi^{-1}=T_{c_2,d_2}^{(i)}.\end{array}\]
Lemma \ref{L: Tau} further implies that
$\phi=\Delta((p,\kappa_q))$ for some $p,q\in\T$, and by \eqref{E: T} we have
\begin{equation}\label{E: Tab}
\begin{array}{lll} T_{qa_1\overline{q},qb_1\overline{q}}^{(j)}=T_{c_1,d_1}^{(j)}
&\text{and}&T_{qa_2\overline{q},qb_2\overline{q}}^{(i)}=T_{c_2,d_2}^{(i)}.\end{array} 
\end{equation}
Now the map $(a,b)\mapsto T^{(k)}_{a,b}$ is a 2-1-map from $\T\times\T$ to the set of all
isometries of $\HH$ with determinant $k$, the preimage of $T^{(k)}_{a,b}$ being $\{\pm(a,b)\}$. Thus \eqref{E: Tab} implies that
\[([qa_1\overline{q},qb_1\overline{q}],[qa_2\overline{q},qb_2\overline{q}])=([c_1,d_1],[c_2,d_2]).\]
Altogether, for each $\phi: A\to B$ there exist $p,q\in\T$ with
\[(p,\kappa_q)\cdot ([a_1,b_1],[a_2,b_2])=([c_1,d_1],[c_2,d_2]),\]
such that $\phi=\mathcal I^{ij}(p,\kappa_q)$. This proves fullness, whereby the proof is complete.
\end{proof}

\begin{Rk} Similarly to the previous case we have, for each $(i,j)\in \Z_2^2$, a description of the coproduct of
two subcategories without it being obvious which objects belong to one or the other. From Proposition 28 in \cite{P}
we can however read off that $\mathcal I^{ij}([a_1,b_1],[a_2,b_2])$ is in $\mathcal D_{1,1,2,4}$
if $\dim[\Im(a_k),\Im(b_k)|k\in\{1,2\}]=1$, and in $\mathcal D_{1,1,1,1,4}$ otherwise. Descriptions of $\mathcal
D_{1,1,2,4}^{ij}$ and $\mathcal D_{1,1,1,1,4}^{ij}$ for each $(i,j)\in\Z_2^2$ can thus be obtained by restricting the functor
$\mathcal I^{ij}$ accordingly.
\end{Rk}

\section{Quasi-Descriptions}
A description of a groupoid $\mathcal C$ takes into account all isomorphisms in $\mathcal C$. As we saw above, for subcategories
of $\mathcal D$ the descriptions involve a large amount of information. In this section we introduce the notion of
a quasi-description, and apply it to the groupoid $\mathcal D$. The idea is that for
classification purposes, it is not necessary to consider all isomorphisms between the objects of a groupoid, as it suffices to
detect whether or not there exists an isomorphism between two objects. In the case of the groupoid
$\mathcal D$, this gives, as we will see, a simpler and more unified approach, from which we will be able to obtain a
classification relatively easily.

\subsection{Preliminaries}
We begin by defining the general set-up.

\begin{Def} A functor $\mathcal F:\mathcal B\to\mathcal C$ between two categories $\mathcal B$ and $\mathcal C$ is said to
\emph{detect non-isomorphic objects} if for any $B,B' \in\mathcal B$,
\[\mathcal F(B)\simeq\mathcal F(B') \Longrightarrow B \simeq B'.\]
\end{Def}

Note that the inverse implication holds for all functors. If $\mathcal B$ and $\mathcal C$ are groupoids, then $\mathcal
F:\mathcal B\to\mathcal C$ detects non-isomorphic objects if and only if each non-empty morphism class in $\mathcal C$ contains at
least one morphism that is in the image of $\mathcal F$. 

\begin{Ex} Any full functor between groupoids detects non-isomorphic objects.\end{Ex}

\begin{Prp}\label{P: Quasi} Let $\mathcal F:\mathcal B\to\mathcal C$ be a dense functor between two categories
$\mathcal B$ and $\mathcal C$ which detects non-isomorphic objects. If an object class $\mathcal S\subseteq \mathcal B$ classifies
$\mathcal B$, then $\mathcal F(\mathcal S)$ classifies $\mathcal C$.
\end{Prp}

Here we say that a subclass $\mathcal S$ of a category $\mathcal B$ \emph{classifies} $\mathcal B$ or \emph{is a classification
of} $\mathcal B$ if $\mathcal S$ is the object class of a skeleton of $\mathcal B$, i.e.\ if $\mathcal S$ is dense in $\mathcal
B$ and different objects in $\mathcal S$ are non-isomorphic.

\begin{proof} If $C\in\mathcal C$, then $C\simeq \mathcal F(B)$ for some $B\in \mathcal B$ by denseness of $\mathcal F$, and
$B\simeq S$ for some $S\in\mathcal S$ as $\mathcal S$ is a classification. Thus $C\simeq\mathcal F(B)\simeq\mathcal F(S)$, whence
$\mathcal F(\mathcal S)$ is dense in $\mathcal C$. Take now $C\neq C'\in\mathcal
F(\mathcal S)$. Then $C=\mathcal F(S)$ and $C'=\mathcal F(S')$ with $S\neq S'\in \mathcal S$. If $C\simeq C'$ in $\mathcal C$,
then $S\simeq S'$ in $\mathcal B$ as $\mathcal F$ detects non-isomorphic objects. This contradicts $\mathcal S$ being a
classification of $\mathcal B$, whence necessarily $C\not\simeq C'$.
\end{proof}

\begin{Rk} Given $\mathcal C$ for which one can provide a category $\mathcal B$ and a dense functor
$\mathcal F:\mathcal B\to\mathcal C$ detecting non-isomorphic objects, the problem of finding a classification of $\mathcal C$
carries
over to that of $\mathcal
B$. However, as the
functor is in general neither faithful nor full, more precise information about the morphisms may be lost. 
\end{Rk}

We now use this to generalize the concept of a description as follows.

\begin{Def} A \emph{quasi-description} of a category $\mathcal C$ is a quadruple $(G,M,\alpha,\mathcal F)$
where $G$ is a group, $M$ is a set, $\alpha: G\times M\to M$ a group action, and $\mathcal F: \phantom{}_GM\to\mathcal C$ a
dense functor which detects non-isomorphic objects. 
\end{Def}

If $\mathcal C$ is a groupoid, then every description of $\mathcal C$ is a quasi-description. In general, the above proposition
and remark imply that the problem of classifying $\mathcal C$ is transferred, as in the case where there
is a description, to the solution of the normal form problem for the action $\alpha$.

Equipped with these tools, we return to composition algebras. Recall that for each $A\in\mathcal D$, we set
\[A_0=\{a\in A|\forall \delta\in\Der(A): \delta(a)=0\}.\]
We define the full subcategories $\mathcal L$ and $\mathcal H$ of $\mathcal D$ with object classes
\[\mathcal L=\{A\in\mathcal D|\dim A_0\leq 1\}\]
and
\[\mathcal H=\{A\in\mathcal D|\dim A_0> 1\}.\]
Thus $\mathcal D=\mathcal L\amalg\mathcal H$. We will treat each of these subcategories separately.
Before doing so, we define the group actions to be used in the
quasi-descriptions. As before we identify $\T$ with $\mathbb S(\HH)$ and $SO_3$ with $\left\{\kappa_q|q\in\T\right\}$. Then $SO_3$
acts on
$\T\times\T$ by
\[\kappa_q\cdot (a,b)=(\kappa_q(a),\kappa_q(b)).\]
When there is no ambiguity, we will simply refer to this as \emph{the action of $SO_3$ on $\T\times\T$}. This defines the groupoid
$_{SO_3}\left(\T\times\T\right)$. Moreover, it induces the action of $SO_3$ on $\left[\T\times\T\right]=(\T\times\T)/\{\pm(1,1)\}$
given by
\[\kappa_q\cdot[a,b]=[\kappa_q(a),\kappa_q(b)].\]
This will in turn be referred to as \emph{the action of $SO_3$ on $\left[\T\times\T\right]$}, and defines the groupoid
$_{SO_3}\left[\T\times\T\right]$. Using this, the groupoid $_{SO_3}\left[\T\times\T\right]^2$ is defined by
\[\kappa_q\cdot([a_1,b_1],[a_2,b_2])=([\kappa_q(a_1),\kappa_q(b_1)],[\kappa_q(a_2),\kappa_q(b_2)]),\]
and the corresponding action will be referred to as \emph{the action of $SO_3$ on $\left[\T\times\T\right]^2$}.

\subsection{Algebras with Low-Dimensional Trivial Submodule}
From Proposition \ref{P: Decomp} we know that the category $\mathcal L$ admits the coproduct decomposition
\[\mathcal L=\mathcal D_{1,7}\amalg\mathcal D_{\{8\}}\amalg \mathcal D_{1,3,4}\amalg\mathcal D_{3+5}.\]
In this section we will obtain a quasi-description of
\[\mathcal L_0=\mathcal D_{1,7}\amalg\mathcal D_{\{8\}}\amalg \mathcal D_{1,3,4},\]
while $\mathcal D_{3,5}$ will be treated in Section 5. A quasi-description of
each double sign component of $\mathcal L_0$ is obtained as follows.

\begin{Thm}\label{T: Quasi1} Let $(i,j)\in \Z_2^2$. The map
\[\mathcal J^{ij}: \phantom{}_{SO_3}\left(\T\times\T\right) \to \mathcal L_0^{ij}\]
defined on objects by
\[(a,b)\mapsto \OO_{K^j\tau_a,K^i\tau_b}\]
and on morphisms by $\kappa_q\mapsto\widehat{\kappa}_q$, is a dense functor detecting non-isomorphic objects.
\end{Thm}

For each $q\in\T$, $\widehat{\kappa}_q\in G_2$ was defined in Section 3 by $(u,v,z)\mapsto (\kappa_q(u),\kappa_q(v),z)$.
The proof is a synthesis of the results and arguments of Sections 3.1-3.3.

\begin{proof} The map $\mathcal J^{ij}$ is well-defined on objects by \cite{P}, Proposition 19, and the fact that the double
sign of $\OO_{K^j\tau_a,K^i\tau_b}$ is $(i,j)$. Regarding morphisms, observe first that identities are
mapped
to identities. Next assume that
$\kappa_q\cdot(a,b)=(c,d)$ for some $q,a,b,c,d\in\T$. From \eqref{E: tau} with $p=1$ we deduce that $\widehat{\kappa}_q\in
G_2$ satisfies                                                                                    
\[\begin{array}{lll}\widehat{\kappa}_qK^j\tau_a\widehat{\kappa}_q^{-1}=K^j\tau_c&\text{and}&\widehat{\kappa}_qK^i\tau_b\widehat{
\kappa}_q^{-1}
=K^i\tau_d.\end {array}\]
Then by Remark \ref{R: U1}, $\mathcal J^{ij}$ maps morphisms to morphisms. Functoriality is clear since
$\kappa_q\kappa_{q'}=\kappa_{qq'}$ and $\widehat\kappa_q\widehat\kappa_{q'}=\widehat\kappa_{qq'}$for any $q,q'\in\T$.

The intersection of the image of $\mathcal J^{ij}$ with $\mathcal D_{1,7}^{ij}$ is dense in $\mathcal
D_{1,7}^{ij}$ by Proposition \ref{P: 1,7}, the intersection with $\mathcal D_{\{8\}}^{ij}$ is dense in
$\mathcal D_{\{8\}}^{ij}$ by Example \ref{E: EM} and the comments following it, and the intersection with $\mathcal
D_{1,3,4}^{ij}$ is dense in $\mathcal D_{1,3,4}^{ij}$ by
Proposition \ref{P: 1,3,4}. Since $\mathcal L_0^{ij}$ is the coproduct of these categories, this proves that $\mathcal J^{ij}$ is
dense.

It remains to be shown that $\mathcal J^{ij}$ detects non-isomorphic objects. Assume thus that $(a,b),(c,d)\in\T\times\T$ satisfy
$A=\OO_{K^j\tau_a,K^i\tau_b}\simeq\OO_{K^j\tau_c,K^i\tau_d}=B$.

If $A,B\in\mathcal D_{1,7}^{ij}$, then by \cite{P}, Proposition 19, $a=b=c=d=1$, whence obviously $(a,b)\simeq(c,d)$
in $_{SO_3}\left(\T\times\T\right)$.

If $(i,j)=(1,1)$ and $A,B\in \mathcal D_{\{8\}}^{ij}$,
then from Proposition 19 in \cite{P} we have $a^2+a=c^2+c=-1$ and $b-a^2=d-c^2=0$. Solving these equations one deduces that
there exist
$w,w'\in\mathbb S\left(1^\bot\right)$ such that
\[\begin{array}{lll} (a,b)=\left(-\frac{1}{2}+\frac{\sqrt 3}{2}w,-\frac{1}{2}-\frac{\sqrt
3}{2}w\right)&\text{and}&(c,d)=\left(-\frac{1}{2}+\frac{\sqrt 3}{2}w',-\frac{1}{2}-\frac{\sqrt 3}{2}w'\right).
  \end{array}
\]
Since $SO_3$ acts transitively on $\mathbb S\left(1^\bot\right)\subset\HH$, there exists $q\in\T$ such that $\kappa_q(w)=w'$,
whence
$(a,b)\simeq(c,d)$.

Finally, if $A,B\in\mathcal D_{1,3,4}^{ij}$, then $A$ and $B$ lie in the image of the functor $\mathcal G^{ij}$ of Proposition
\ref{P: 1,3,4}. Since that
is an equivalence of categories, there exist $p,q\in\T$ such that
\[(c,d)=(pqa\overline{pq},pqb\overline{pq}),\]
whence $\kappa_{pq}\cdot(a,b)=(c,d)$, and $(a,b)\simeq(c,d)$ in $_{SO_3}\left(\T\times\T\right)$. 
\end{proof}

A transversal for the action of $SO_3$ on $\T\times\T$ will be given in the final section. 

\subsection{Algebras with High-Dimensional Trivial Submodule}
Proposition \ref{P: Decomp} implies that the category $\mathcal H$ decomposes as
\[\mathcal H=\mathcal D_{1,1,6}\amalg\mathcal D_{1,1,2,4}\amalg \mathcal D_{1,1,1,1,4}\amalg \mathcal D_{1,1,3,3}.\]
We will give a quasi-description of
\[\mathcal H_0=\mathcal D_{1,1,6}\amalg\mathcal D_{1,1,2,4}\amalg \mathcal D_{1,1,1,1,4},\]
while $\mathcal D_{1,1,3,3}$ will be treated in Section 5. The quasi-description is constructed in two steps,
of which the following lemma is the first.

\begin{Lma}\label{L: J} Let $(i,j)\in \Z_2^2$. The map
\[\mathcal K^{ij}: \phantom{}_{SO_3}\left[\T\times\T\right]^2 \to \mathcal D^{ij}\]
defined on objects by
\[([a_1,b_1],[a_2,b_2])\mapsto \OO_{T_{a_1,b_1}^{(j)},T_{a_2,b_2}^{(i)}}\]
and on morphisms by $\kappa_q\mapsto\widehat{\kappa}_q$, is a functor.
\end{Lma}

\begin{proof} The map $\mathcal K^{ij}$ is well-defined on objects by \cite{P}, Proposition 28, and by double sign
considerations, and maps identities to identities. If 
\[\kappa_q\cdot([a_1,b_1],[a_2,b_2])=([c_1, d_1],[c_2,d_2]),\]
then
applying \eqref{E: T} with $p=1$ we see that
$\widehat{\kappa}_q\in G_2$ satisfies
\[\begin{array}{lll}\widehat{\kappa}_q T^{(j)}_{a_1,b_1}\widehat{\kappa}_q^{-1}=T^{(j)}_{c_1,d_1}&\text{and}&\widehat{\kappa}_q
T^{(i)}_{a_2,b_2}\widehat{\kappa}_q^{-1}=T^{(i)}_{c_2,d_2}.\end{array}\]
Then Lemma \ref{L: G2} (with $U=\HH^\bot$) implies that
$\widehat{\kappa}_q:\OO_{T_{a_1,b_1}^{(j)},T_{a_2,b_2}^{(i)}}\to\OO_{T_{c_1,d_1}^{(j)},T_{c_2,d_2}^{(i)}}$,
and thus $\mathcal K^{ij}$ is well-defined on morphisms. Functoriality follows from the fact that
$\kappa_q\kappa_{q'}=\kappa_{qq'}$ and $\widehat\kappa_q\widehat\kappa_{q'}=\widehat\kappa_{qq'}$for any $q,q'\in\T$.
\end{proof}

We wish to use these functors to obtain quasi-descriptions of $\mathcal H_0$. However, for each $(i,j)\in\Z_2^2$, the image of
$\mathcal K^{ij}$ is not contained in $\mathcal H_0$. This is made precise in the following remark.

\begin{Rk}\label{R: Remark}
By Proposition 28 in \cite{P},
$\OO_{T_{a_1,b_1}^{(j)},T_{a_2,b_2}^{(i)}}=\mathcal K^{ij}([a_1,b_1],[a_2,b_2])$ belongs to $\mathcal H_0$ if and only if 
\[\{a_1,b_1,a_2,b_2\}\nsubseteq\{\pm1\}.\]
The set of all $[a_1,b_1],[a_2,b_2]$ such that $\{a_1,b_1,a_2,b_2\}\subseteq\{\pm1\}$ is pointwise fixed under the action of
$SO_3$. Call the complement of this set $S$. Then the action of $SO_3$ on $\left[\T\times\T\right]^2$ induces an action on $S$,
giving
rise
to the groupoid $_{SO_3}S$.
\end{Rk}

We will use the groupoid $_{SO_3}S$ in our quasi-descriptions as follows.

\begin{Thm}\label{T: Quasi2} Let $(i,j)\in \Z_2^2$. The map
\[\mathcal K_*^{ij}: \phantom{}_{SO_3}S \to \mathcal H_0^{ij}\]
defined on objects by
\[([a_1,b_1],[a_2,b_2])\mapsto \OO_{T_{a_1,b_1}^{(j)},T_{a_2,b_2}^{(i)}}\]
and on morphisms by $\kappa_q\mapsto\widehat{\kappa}_q$, is a dense functor, detecting non-isomorphic objects.
\end{Thm}

The proof builds on the results of Sections 3.2 and 3.4.

\begin{proof} The image of $\mathcal K_*^{ij}$ is in $\mathcal H_0^{ij}$ by Remark \ref{R: Remark}. Therefore
$\mathcal K_*^{ij}$ is a well-defined functor, being the restriction of the functor $\mathcal K^{ij}$ from Lemma \ref{L: J}.
We will now show that $\mathcal K_*^{ij}$ is dense and detects non-isomorphic objects by considering its image in
$\mathcal D_{1,1,6}$ and $\mathcal D_{1,1,2,4}\amalg \mathcal D_{1,1,1,1,4}$ separately, starting with the
first.

If $A\in\mathcal D_{1,1,6}^{ij}$, then $A\simeq\OO_{\lambda_{e_1}^{(j)},\lambda_{e_2}^{(i)}}$ for
some
$e_1,e_2\in\C\subset\OO$,
by Proposition \ref{P: 1,1,6}. To prove that $\OO_{\lambda_{e_1}^{(j)},\lambda_{e_2}^{(i)}}$ is the image of some
$([a_1,b_1],[a_2,b_2])\in S$, note that there exist $\alpha_1,\alpha_2\in[0,\pi)$ such that
\[\begin{array}{ll} e_1=\cos(\pi j+2\alpha_1)+u\sin(\pi j+2\alpha_1),& e_2=\cos(\pi i+2\alpha_2)+u\sin(\pi
i+2\alpha_2).\end{array}\]
Take such $\alpha_1$ and $\alpha_2$ and set, for each $m\in\{1,2\}$, $a_m=\cos(\alpha_m)+u\sin(\alpha_m)$. Moreover set
\[\begin{array}{lll} b_1=(-1)^ja_1&\text{and}& b_2=(-1)^ia_2.\end{array}\]
By straightforward computations we see that  $\lambda_{e_1}^{(j)}=T_{a_1,b_1}^{(j)}$ and $\lambda_{e_2}^{(i)}=T_{a_2,b_2}^{(i)}$.
If
$\{a_1,b_1,a_2,b_2\}\subset\langle 1\rangle$, then $\alpha_1=\alpha_2=0$, and $(e_1,e_2)=((-1)^j,(-1)^i)$, which by Proposition
\ref{P:
1,1,6} contradicts that $\OO_{\lambda_{e_1}^{(j)},\lambda_{e_2}^{(i)}}\in\mathcal D_{1,1,6}^{ij}$.
Thus
$\{a_1,b_1,a_2,b_2\}\nsubseteq\langle 1\rangle$, whence $([a_1,b_1],[a_2,b_2])\in S$ and
$\OO_{\lambda_{e_1}^{(j)},\lambda_{e_2}^{(i)}}=\mathcal
K_*^{ij}([a_1,b_1],[a_2,b_2])$.

Assume next that $\OO_{T_{a_1,b_1}^{(j)}, T_{a_2,b_2}^{(i)}},\OO_{T_{c_1,d_1}^{(j)},T_{c_2,d_2}^{(i)}}\in\mathcal
D_{1,1,6}^{ij}$ are isomorphic. We will show that $([a_1,b_1],[a_2,b_2])\simeq([c_1,d_1],[c_2,d_2])$ in $_{SO_3}S$. By
Proposition 28
in \cite{P}, $\OO_{T_{a_1,b_1}^{(j)}, T_{a_2,b_2}^{(i)}}\in\mathcal D_{1,1,6}^{ij}$ implies that
the
subalgebra of
$\OO$ generated by $\{a_1,b_1,a_2,b_2\}$ is two-dimensional, and that
\[\begin{array}{lll} b_1=(-1)^ja_1&\text{and}& b_2=(-1)^ia_2.\end{array}.\]
Hence there exist $w\in\mathbb S\left(1^\bot\right)$ and $\alpha_1,\alpha_2\in[0,2\pi)$ such that for each $m\in\{1,2\}$
\[\begin{array}{lll}
a_m= \cos(\alpha_m)+w\sin(\alpha_m) &\text{and}&b_m=\cos(\pi k_m+\alpha_m)+w\sin(\pi k_m + \alpha_m),
  \end{array}
\]
where $k_1=j$ and $k_2=i$. Now, $w=\kappa_q(u)$ for some $q\in \T$. Therefore\linebreak
$([a_1,b_1],[a_2,b_2])\simeq([a_1',b_1'],[a_2',b_2'])$ in $_{SO_3}S$, where
\[\begin{array}{lll}
a_m'= \cos(\alpha_m)+u\sin(\alpha_m) &\text{and}&b_m'=\cos(\pi k_m+\alpha_m)+u\sin(\pi k_m + \alpha_m),
  \end{array}
\]
Likewise,  $([c_1,d_1],[c_2,d_2])\simeq([c_1',d_1'],[c_2',d_2'])$ in $_{SO_3}S$, with
\[\begin{array}{lll}
c_m'= \cos(\beta_m)+u\sin(\beta_m) &\text{and}&d_m'=\cos(\pi k_m+\beta_m)+u\sin(\pi k_m + \beta_m),
  \end{array}
\]
for some $\beta_1,\beta_2\in[0,2\pi)$. For each $m\in\{1,2\}$ we set
\[e_m=\cos(\pi k_m+2\alpha_m)+u\sin(\pi k_m+2\alpha_m)\]
and
\[f_m=\cos(\pi k_m+2\beta_m)+u\sin(\pi k_m+2\beta_m).\]
Then straightforward computations show that 
\[\begin{array}{lllll}\lambda_{e_1}^{(j)}=T_{a_1',b_1'}^{(j)},&\lambda_{e_2}^{(i)}=T_{a_2',b_2'}^{(i)},&\lambda_{f_1}^{(j)}=T_{
c_1', d_1'}^{(j)} , & \text{and}&\lambda_{f_2}^{(i)}=T_{c_2',d_2'}^{(i)}.\end{array}\]
Thus $\OO_{T_{a_1,b_1}^{(j)}, T_{a_2,b_2}^{(i)}}\simeq
\OO_{T_{c_1,d_1}^{(j)},T_{c_2,d_2}^{(i)}}$ implies that $\OO_{\lambda_{e_1}^{(j)},
\lambda_{e_2}^{(i)}}\simeq\OO_{\lambda_{f_1}^{(j)},\lambda_{f_2}^{(i)}}$. Using Proposition
\ref{P: 1,1,6} this gives
\[\begin{array}{lll}(e_1,e_2)=(f_1,f_2)&\text{or}&(e_1,e_2)=\left(\overline{f_1},\overline{f_2}\right).\end{array}\]
In both cases there exists $q\in\T$ such that $\kappa_q\cdot([a_1',b_1'],[a_2',b_2'])=([c_1',d_1'],[c_2',d_2'])$. Altogether this
implies that $([a_1,b_1],[a_2,b_2])\simeq([c_1,d_1],[c_2,d_2])$ in $_{SO_3}S$, which was what we desired. 

Thus the part of the proof concerned with $\mathcal D_{1,1,6}$ is complete, and we turn to the category
$\mathcal D_{1,1,2,4}\amalg \mathcal D_{1,1,1,1,4}$. By Lemma \ref{L: Dense2}, the image $\mathcal
K_*^{ij}([S_{ij}])\subseteq\mathcal K_*^{ij}(S)$ is dense in this category. Assume next
that 
\[\OO_{T_{a_1,b_1}^{(j)}, T_{a_2,b_2}^{(i)}},\OO_{T_{c_1,d_1}^{(j)},T_{c_2,d_2}^{(i)}}\in \mathcal K_*^{ij}(S)\]
are isomorphic.
By Proposition 28 in \cite{P}, these belong to $\mathcal D_{1,1,2,4}^{ij}\amalg \mathcal D_{1,1,1,1,4}^{ij}$
if
and only if $([a_1,b_1],[a_2,b_2]),([c_1,d_1],[c_2,d_2])\in [S_{ij}]$. Then since the functor $\mathcal I^{ij}$ of Proposition
\ref{P: 1114} is full, the algebras being isomorphic implies that there exists $q\in\T$ such that
\[([qa_1\overline{q},qb_1\overline{q}],[qa_2\overline{q},qb_2\overline{q}])=([c_1,d_1],[c_2,d_2]),\]
whence $([a_1,b_1],[a_2,b_2])\simeq([c_1,d_1],[c_2,d_2])$ in $_{SO_3}S$. This completes the part of the proof concerned with
$\mathcal D_{1,1,2,4}\amalg \mathcal
D_{1,1,1,1,4}$. 

Since $\mathcal H_0^{ij}$ is the coproduct of the two subcategories treated above, it follows that $\mathcal K_*^{ij}$
is
dense and detects non-isomorphic objects, and the proof is complete.  
\end{proof}

The category $\phantom{}_{SO_3}\left[\T\times\T\right]^2$ will be classified in the final section.

\section{Algebras with Decomposition $\{3,5\}$ or $\{1,1,3,3\}$}
We now come to those subcategories of $\mathcal D$ which fall outside the above treatment via quasi-descriptions. One common
property of the algebras in these subcategories is that their derivation algebras are of type $\mathfrak {su}_2$. Moreover, as was
constructively demonstrated in \cite{P}, they can all be expressed as isotopes of algebras having a derivation algebra of type
$\mathfrak {su}_3$. In this section we first consider $\mathcal L\setminus \mathcal L_0=\mathcal D_{3,5}$, consisting of certain
isotopes of Okubo algebras. We show that this category consists of three
isomorphism classes, one in each double sign different from $(-,-)$. (Recall that Okubo algebras form the unique isomorphism class
$\mathcal D_{\{8\}}=\mathcal D_{\{8\}}^{11}$.) We thereafter consider $\mathcal H\setminus \mathcal H_0=\mathcal D_{1,1,3,3}$,
which is exhausted by isotopes of algebras in $\mathcal D_{1,1,6}$. We show that these are classified by
twelve two-parameter families of algebras, three in each double sign. (Recall here that $\mathcal D_{1,1,6}$ itself was classified
by four two-parameter families, one in each double sign.)

\subsection{The Category $\mathcal D_{3,5}$}
The isotopes in $\mathcal D_{3,5}$ were constructed in \cite{P} as follows. Let $P\in\mathcal A_8$ be an
Okubo algebra, and $e\in P$ a non-zero idempotent. Then $C=P_{(R_e^P)^{-1},(L_e^P)^{-1}}$ is unital with unity $e$, and we denote
its multiplication by juxtaposition and its standard involution by $J$. Since $P$ is an Okubo algebra, it is known from \cite{EP0}
that $(L_e^P)^{-1}=R_e^P=J\tilde{\tau}$ for some automorphism $\tilde{\tau}$ of $C$
of order three, such that the fixed point set of $\tilde\tau$ is a unital four-dimensional subalgebra $Q$ of $C$. Choose now $x\in
Q^\bot$ of
norm
one. Then there is a unique $y\in Q$ such that $\tilde\tau(x)=xy$, and we choose $a\in Q$ of norm one orthogonal to $e$ and $y$.
For any such choice of $x$ and $a$ we call the span of $\{a,x,ax\}$ a \emph{special subspace} of $P$. 

Let now $P\in\mathcal A_8$ be an Okubo algebra and $S$ a special subspace of $P$. We may then construct, for each
$(i,j)\in\Z_2$, the
isotope $P_{\sigma_S^{1-j},\sigma_S^{1-i}}$, where $\sigma_S$ acts as $-1$ on $S$ and as identity on $S^\bot$. (If $(i,j)=(1,1)$
this is just
$P$ itself.) It was proven in \cite{P} that $P_{\sigma_S^j,\sigma_S^i}\in\mathcal D_{3,5}$ for each
$(i,j)\neq(1,1)$ and,
conversely, that each $A\in\mathcal D_{3,5}$ is isomorphic to $P_{\sigma_S^{1-j},\sigma_S^{1-i}}$ for an Okubo algebra
$P\in\mathcal A_8$, a special subspace $S$ of $P$, and a pair $(i,j)\neq (1,1)$. Note that the double sign of
$P_{\sigma_S^{1-j},\sigma_S^{1-i}}$ is $((-1)^i,(-1)^j)$, since $P$ has double sign $(-,-)$.

\begin{Ex}\label{Ex: Okubo}
The division Okubo algebra $P^{11}$ was defined in Example \ref{E: EM} by 
\[P^{11}=\OO_{K\tau,K\tau^{-1}},\]
where $\tau=\tau_{\left(\sqrt 3 u-1\right)/2}$. One can check that $W=\langle v,z,vz\rangle$ is a special subspace of $P^{11}$.
Thus for
each $(i,j)\in\Z_2^2$ different from $(1,1)$, the isotope
\[P^{ij}=\OO_{K\tau\sigma_W^{1-j},K\tau^{-1}\sigma_W^{1-i}}\]
belongs to $\mathcal D_{3,5}$.
\end{Ex}

The next result shows that the choices of $P$ and $S$ above are immaterial, and classifies $\mathcal D_{3,5}$. Recall that we
have fixed a Cayley triple $(u,v,z)$ in $\OO$.

\begin{Thm} Let $(i,j)\in\Z_2^2$ and let $P^{ij}$ be defined as in Example
\ref{Ex: Okubo}. The category $\mathcal D_{3,5}^{ij}$ is classified by $\{P^{ij}\}$, if $(i,j)\neq
(1,1)$,
and $\mathcal D_{3,5}^{11}=\emptyset$.
\end{Thm}

\begin{proof} The emptiness of $\mathcal D_{3,5}^{11}$ and the fact that
$P^{ij}\in\mathcal D_{3,5}^{ij}$ if $(i,j)\neq(1,1)$ follow from the above discussion.

Let $(i,j)\neq(1,1)$ and take $A\in\mathcal D_{3,5}^{ij}$. It remains to be shown that $A\simeq
P^{ij}$. By
the results from \cite{P} recalled above, there exist $P$, $e$, $\tilde\tau$, $Q$, $a$ and $x$ as above, such that $S=\langle
a,x,ax \rangle$ 
(where juxtaposition is multiplication in $C=P_{(R_e^P)^{-1},(L_e^P)^{-1}}$) is a special subspace of $P$, and $A\simeq
P_{\sigma_S^{1-j},\sigma_S^{1-i}}$. Now
\[P=(P_{(R_e^P)^{-1},(L_e^P)^{-1}})_{R_e^P,L_e^P}=C_{R_e^P,L_e^P}=C_{J\tilde\tau, J\tilde\tau^{-1}
},\]
where $J$ is the standard involution on the unital algebra $C$ . By Lemma \ref{L: Canonical}, there is an
isomorphism $\psi: C\to \OO$, mapping $e$ to $1$ and $Q$ to $\HH$, whence Lemma \ref{L: Juggling} supplies the isomorphism
\[\psi: P\to \OO_{\psi J\tilde\tau \psi^{-1},\psi J\tilde\tau^{-1} \psi^{-1}}.\]
Since $x\bot Q$ and has norm one, we have $\psi(x)\in\mathbb S\left(\HH^\bot\right)$. Recall that $y\in Q$ is uniquely defined by
$\tilde\tau(x)=xy$. Since $\tilde\tau$ has order three, we necessarily have $y^3=e$ in $C$, and hence $\psi(y)^3=1$ in $\OO$.
Thus there exists $w\in\mathbb S(\HH)$ such that $w\bot 1$ and $\psi(y)=-1/2+w\sqrt{3}/2$. Then $a\bot\langle e,y\rangle$ implies
that $\psi(a)\in\mathbb S(\HH)\cap \langle 1,w\rangle^\bot$,
whence $(w,\psi(a),\psi(x))$ is a Cayley triple. Therefore there is a unique $\rho\in G_2$ mapping
$(w,\psi(a),\psi(x))$ to $(u,v,z)$. Applying Lemma \ref{L: Juggling} to the composition $\phi=\rho\psi$ gives the
isomorphism
\[\phi: P\to \OO_{\phi J\tilde\tau\phi^{-1},\phi J\tilde\tau^{-1} \phi^{-1}}.\]
Since $\phi$ is also an isomorphism from $C$ to $\OO$ mapping $Q$ to $\HH$ and $y$ to $-1/2+u\sqrt{3}/2$, it follows from the
properties of $\tilde\tau$ that
$\phi\tilde\tau \phi^{-1}$ is an automorphism of $\OO$ defined by fixing $\HH$ pointwise and mapping $z=\phi(x)$ to
$\phi(xy)=z(\sqrt{3}u-1)/2$. Thus $\phi\tilde\tau \phi^{-1}=\tau$. Moreover $\phi J\phi^{-1}=K$. Therefore
\[\OO_{\psi J\tilde\tau \psi^{-1},\psi J\tilde\tau^{-1} \psi^{-1}}= \OO_{K\tau,K\tau^{-1}}=P^{11}.\]
A final application of Lemma \ref{L: Juggling} provides the isomorphism
\[\phi:  P_{\sigma_S^{1-j},\sigma_S^{1-i}}\to \OO_{K\tau\phi\sigma_S^{1-j}\phi^{-1},K\tau^{-1}\phi\sigma_S^{1-i}\phi^{-1}}.\]
Now $\phi\sigma_S\phi^{-1}$ acts as $-1$ on $\phi(S)$ and as identity on $\phi(S)^\bot$. Since $\phi(S)$ is the subspace
$W=\langle v,z,vz\rangle$ from Example \ref{Ex: Okubo}, this proves that $P_{\sigma_S^{1-j},\sigma_S^{1-i}}\simeq P^{ij}$. Hence
$A\simeq P^{ij}$, and the proof
is complete.
\end{proof}

In the above proof we in fact showed anew that each Okubo algebra $P\in\mathcal A_8$ is isomorphic to $P^{11}$. The reason for
doing so, rather than merely quoting the fact, is that we needed to track the behaviour of the special subspace $S\subset P$ under
the isomorphism.

\begin{Rk}
Each $P^{ij}$ is an isotope of $P^{11}=\mathcal J^{11}\left(t,t^2\right)$, where $t=\left(\sqrt 3 u-1\right)/2$, and the
functors $\mathcal J^{ij}$ were defined in Section 4.2. However, it is not true, for $(i,j)\neq(1,1)$, that $P^{ij}=\mathcal
J^{ij}(t,t^2)$. If $(i,j)\neq(1,1)$, then $\mathcal J^{ij}\left(t,t^2\right)=\mathcal G^{ij}\left(t,t^2\right)$,
where $\mathcal G^{ij}$ was defined in Proposition \ref{P: 1,3,4}. Therefore, $\mathcal J^{ij}\left(t,t^2\right)$ belongs to
$\mathcal D_{1,3,4}$. Indeed, $\left(t,t^2\right)$ is in the domain
$\phantom{}_{\T\rtimes SO_3}\left(\T\times\T\right)_{ij}$ of $\mathcal G^{ij}$ if and only if $(i,j)\neq(1,1)$.
\end{Rk}

\subsection{The Category $\mathcal D_{1,1,3,3}$}
The below treatment of $\mathcal D_{1,1,3,3}$ is, as we will see, quite technical in nature. A dense subset of $\mathcal
D_{1,1,3,3}$ was given in \cite{P}. We begin by giving a refined version of it, after introducing some notation which we will use
throughout this section.

\begin{Rk}\label{R: Notation} Let $a\in \mathbb S(\OO)$. We write $L_a=L_a^\OO$ and $R_a=R_a^\OO$ for the maps $x\mapsto
ax$ and
$x\mapsto xa$,
respectively. Since $\OO$ is alternative, $L_aR_a=R_aL_a$ and $L_aR_{\overline{a}}=R_{\overline{a}}L_a$, and we write $B_a$ for
the bimultiplication $L_aR_a$ and $C_a$ for the conjugation $L_aR_{\overline{a}}$. (Note that $\overline{a}=a^{-1}$.)
We further write $\sigma_a$ for the reflection in the hyperplane $a^\bot$, i.e.\ $\sigma_a$ is the linear map defined by $a\mapsto
-a$ and $x\mapsto x$ for each $x\bot a$. As usual we have fixed a Cayley triple $(u,v,z)$ in $\OO$.
\end{Rk}

We moreover define, for each $\theta,\gamma\in\R$ and each $(k_1,k_2)\in\Z_2^2$, the map
\[G_{\theta,\gamma}^{k_1,k_2}=B_{\cos\theta+u\sin\theta}\sigma_u^{k_1}(\sigma_{uv}\sigma_{uz}\sigma_{vz\sin\gamma-(uv)z\cos\gamma}
)^ {k_2}:\OO\to\OO.\]
Note that if $k_2=0$, then this map has eigenvalue 1 with eigenspace $\C ^\bot$, and coincides with the map
$\lambda_t^{(k_1)}$ from Section 3.2 with $t=\cos 2\theta+u\sin 2\theta$. If $k_2=1$, then 1 is an
eigenvalue with
eigenspace $W_\gamma=\langle v,z,vz\cos\gamma+(uv)z\sin\gamma\rangle $ and $-1$ is an eigenvalue with eigenspace $uW_\gamma$. 
In all cases, the matrix of $G_{\theta,\gamma}^{k_1,k_2}$ is $2\times2$-block diagonal in the basis $(1,u,v,uv,z,uz,vz,(uv)z)$ of
$\OO$. Note also that since $B_a=B_{-a}$ we have $G_{\theta,\gamma}^{k_1,k_2}=G_{\theta+\pi n,\gamma}^{k_1,k_2}$ for all $n\in\Z$.

\begin{Lma}\label{L: Pe} Let $A\in\mathcal D_{1,1,3,3}$. Then there exist
$(i_1,i_2),(j_1,j_2)\in\Z_2^2$ with $i_2=1$ or $j_2=1$ and $\alpha,\beta,\gamma\in[0,\pi)$ such that $A\simeq
\OO_{f,g}$ with
\begin{equation}\label{E: f}
\begin{array}{lll}f=G_{\alpha,\gamma}^{j_1,j_2}& \text{and}& g=G_{\beta,\gamma}^{i_1,i_2}.\end{array}
\end{equation}
Moreover $((-1)^{i_1+i_2},(-1)^{j_1+j_2})$ is the double sign of $A$.
\end{Lma}

\begin{proof} Let $A\in\mathcal D_{1,1,3,3}$. By Proposition 39 in \cite{P} there exists
$C\in\mathcal A_8$ with unity $e$, an element $x\bot e$ and a three-dimensional subspace $S\subset \langle e,x\rangle ^\bot$ with
$S\bot xS$, such that $A\simeq C_{f',g'}$, where, $f'$ and $g'$ map each of $\langle e,x\rangle $, $S$ and $xS$ isometrically onto
itself, act as
identity on $S$, and satisfy
\[\begin{array}{lll} f'|_{xS}=(-1)^k\Id|_{xS}&\text{and}& g'|_{xS}=(-1)^l\Id|_{xS}
  \end{array}
\]
for some $(k,l)\in\Z_2^2$ with $k=1$ or $l=1$. Let $(s_1,s_2,s_3)$ be an orthonormal basis of $S$. The condition that $S\bot xS$
implies in particular that $xs_1\bot s_2$. If $\psi:C\to\OO$ is an isomorphism, then the image of $(x,s_1,s_2)$ is therefore a
Cayley triple. Define $\rho\in G_2$ by $(\psi(x),\psi(s_1),\psi(s_2))\mapsto (u,v,z)$, and set $\phi=\rho\psi$. Since $\phi$
is orthogonal, it maps $s_3$ into $\langle vz,(uv)z\rangle $, and for some $\gamma\in [0,\pi)$ we have
$\langle \phi(s_3)\rangle =\langle vz\cos\gamma+(uv)z\sin\gamma\rangle $. Thus $\phi f \phi^{-1}$ maps each of $\langle 1,u\rangle
$, $W_\gamma$ and
$uW_\gamma$ isometrically onto itself, acting as the identity on $W_\gamma$, and as $(-1)^k$ times the identity on $uW_\gamma$.
Setting
$f=\phi f'\phi^{-1}$, it follows that $f$ satisfies \eqref{E: f} for some suitable $\alpha\in[0,\pi)$ and $j_1\in\Z_2$, and with
$j_2=k$. By analogous arguments, $g=\phi g'\phi^{-1}$ satisfies \eqref{E: f} for some $\beta\in[0,\pi)$ and
$i_1\in\Z_2$, and with $i_2=l$. Lemma \ref{L: Juggling} then implies that
$C_{f',g'}\simeq\OO_{f,g}$. The statement about the double sign follows from Remark \ref{R: Double Sign}.
\end{proof}

Observe that the converse is false, i.e.\ there exist $f$ and $g$ satisfying \eqref{E: f} with suitable
parameters, such that $\OO_{f,g}\notin \mathcal D_{1,1,3,3}$. We will return to this matter later.

We now have twelve 3-parameter families of algebras to consider. In order to simplify the classification problem, Alberto Elduque
(personal communication, October 2013) showed that the parameter $\gamma$ is superfluous, and further that each $A\in\mathcal
D_{1,1,3,3}$ is isomorphic to $\OO_{f,g}$ for some $f,g\in O_8$ fixing the element 1 and
the subspace $\langle u\rangle $ and determined on $\langle 1,u\rangle ^\bot$ by one parameter each. The precise content of his
result is presented in the next proposition. Our proof however differs from his.  For coherence with our approach we express the
result in terms of the conjugation map defined in Remark \ref{R: Notation}. As in the previous section we set $W=\langle
v,z,vz\rangle$ and write $\sigma_W$ for $\sigma_{uv}\sigma_{uz}\sigma_{(uv)z}$. We define, for each $\theta\in\R$ and each
$(k_1,k_2)\in\Z_2^2$, the map
\[F_\theta^{k_1,k_2}=C_{\cos\theta+u\sin\theta}\sigma_u^{k_1}\sigma_W^{k_2}:\OO\to\OO.\]
In the basis $(1,u,v,uv,z,uz,vz,(uv)z)$, this map has the block-diagonal matrix
\[\left(\begin{array}{lllll}
1&&&&\\
&k_1&&&\\
&&\widehat{2\theta}_{k_2}&&\\
&&&\widehat{2\theta}_{k_2}&\\
&&&&\widehat{-2\theta}_{k_2}\\
        \end{array}\right),\]
where for each $\zeta\in\R$ and each $k\in\Z_2$,
\[\widehat{\zeta}_k=\left(\begin{array}{rr}
\cos\zeta&-(-1)^k\sin\zeta\\
\sin\zeta&(-1)^k\cos\zeta\\
        \end{array}\right).\]

\begin{Prp}\label{P: AE} Let $A\in\mathcal D_{1,1,3,3}$. Then
$A\simeq \OO_{F_{\xi}^{j_1,j_2},F_{\eta}^{i_1,i_2}}$ for some $\xi,\eta\in\R$ and $(i_1,i_2),(j_1,j_2)\in\Z_2^2$ with
$i_2=1$ or $j_2=1$.
\end{Prp}

\begin{proof} By Lemma \ref{L: Pe} there exist $\alpha,\beta,\gamma\in[0,\pi)$ such that
$A\simeq\OO_{G_{\alpha,\gamma}^{j_1,j_2},G_{\beta,\gamma}^{i_1,i_2}}$. 
Let $c=\cos(\gamma/3)-u\sin(\gamma/3)$, and define $\rho\in G_2$ by
$(u,v,z)\mapsto(u,cv,cz)$, which is clearly a
Cayley triple. Let $t=\cos\theta+u\sin\theta$ and $s=\cos\zeta+u\sin\zeta$ for $\theta,\zeta\in\R$ to be chosen later, and set
$\phi=L_tR_s\rho$. Then 
\[(\phi_1,\phi_2)=(B_tR_{\overline{s}}\rho,L_{\overline{t}}B_s\rho)\]
is a pair of triality components for $\phi$, and $A\simeq \OO_{\phi_1 f\phi^{-1},\phi_2 g\phi^{-1}}$ by Proposition \ref{P: Ca}.
We now choose $\theta$ and $\zeta$ so that
\begin{equation}\label{E: i1j1}
(3\theta/2,3\zeta/2)=\left\{ \begin{array}{ll} (\alpha+2\beta,2\alpha+\beta)& \text{if\ } (i_1,j_1)=(0,0),\\
                                               (-\alpha,-2\alpha-3\beta)& \text{if\ } (i_1,j_1)=(0,1),\\
                                               (-3\alpha-2\beta,-\beta)& \text{if\ } (i_1,j_1)=(1,0),\\
                                               (-\alpha,-\beta)& \text{if\ } (i_1,j_1)=(1,1).
  \end{array}\right.
\end{equation}

Then by direct computations one verifies that 
\[\begin{array}{lll}\phi_1f\phi^{-1}=F_{\xi}^{j_1,j_2}&
\text{and}& \phi_2g\phi^{-1}=F_{\eta}^{i_1,i_2},\end{array}\]
where $\xi,\eta\in\R$ are given by
\begin{equation}\label{E: i2j2}
(2\xi,2\eta)=\left\{ \begin{array}{ll} (\theta-2\gamma/3,\zeta-2\theta) & \text{if\ } (i_2,j_2)=(0,1),\\
                                               (2\zeta-\theta,-\zeta-2\gamma/3) & \text{if\ } (i_2,j_2)=(1,0),\\
                                               (\theta-2\gamma/3,-\zeta-2\gamma/3) & \text{if\ } (i_2,j_2)=(1,1).
\end{array}\right.
\end{equation}

The proof is complete.
\end{proof}

\begin{Rk}\label{R: Reverse} Let $(i_1,i_2),(j_1,j_2)\in\Z_2^2$ with $i_2=1$ or $j_2=1$. For any $\xi,\eta\in\R$ and any
$\gamma\in[0,\pi)$, we can solve for $\alpha$ and $\beta$ in \eqref{E: i1j1} and \eqref{E: i2j2}. In particular it follows that
for any $\xi,\eta\in\R$ there exist $\alpha,\beta\in[0,\pi)$ such that $\OO_{F_{\xi}^{j_1,j_2},F_{\eta}^{i_1,i_2}}$ is
isomorphic
to $\OO_{G_{\alpha,0}^{j_1,j_2},G_{\beta,0}^{i_1,i_2}}$. In other words, the parameter $\gamma$ is superfluous.
\end{Rk}

The above proposition and remark together show that for each $A\in\mathcal D_{1,1,3,3}$ there exist $\alpha,\beta\in[0,\pi)$ and
$(i_1,i_2),(j_1,j_2)\in\Z_2^2$ with $i_2=1$ or $j_2=1$ such that $A\simeq\OO_{G_{\alpha,0}^{j_1,j_2},G_{\beta,0}^{i_1,i_2}}$. The
following lemma provides a converse to this statement.

\begin{Lma}\label{L: Exclusion} Let $(i_1,i_2),(j_1,j_2)\in\Z_2^2$ with $i_2=1$ or $j_2=1$, and let $\alpha,\beta\in[0,\pi)$. Then
$\OO_{G_{\alpha,0}^{j_1,j_2},G_{\beta,0}^{i_1,i_2}}\in\mathcal D_{1,1,3,3}$ if and only if 
\[(\cos2\alpha,\cos2\beta)\neq((-1)^{j_1+j_2},(-1)^{i_1+i_2}).\]
\end{Lma}

\begin{proof} Applying Proposition 25 in \cite{P} and its proof to the present setting implies the ``if''-part of the
statement, while the converse is established in Examples 22--24 in \cite{P}.
\end{proof}

Note that for each $(i_1,i_2),(j_1,j_2)\in\Z_2^2$ with $i_2=1$ or $j_2=1$, there is exactly one pair $(\alpha,\beta)\in[0,\pi)^2$
such
that $\OO_{G_{\alpha,0}^{j_1,j_2},G_{\beta,0}^{i_1,i_2}}\notin\mathcal D_{1,1,3,3}$.

The reason for taking the seemingly complicated detour via algebras
of the form $\OO_{F_{\xi}^{j_1,j_2},F_{\eta}^{i_1,i_2}}$ is that isomorphisms between these algebras are easier to construct
explicitly. The following remark provides the background.

\begin{Rk}\label{R: A0} If $A=\OO_{F_\xi^{j_1,j_2},F_\eta^{i_1,i_2}}\in\mathcal D_{1,1,3,3}$,
then $A_0=\C$ as a vector subspace. Indeed, we know from
\cite{P} that this is the case for the algebras in $\mathcal D_{1,1,3,3}$ of the form
$\OO_{G_{\alpha,\gamma}^{j_1,j_2},G_{\beta,\gamma}^{i_1,i_2}}$, and
the isomorphisms in the proof of Proposition \ref{P: AE}
all map $\C$ to itself.  Thus $A_0\in\mathcal A_2$, and by construction, the double sign of $A_0$ is
$((-1)^{i_1},(-1)^{j_1})$. It is known that each algebra in $\mathcal A_2$ has precisely three
non-zero idempotents if its double sign is $(-,-)$, and a unique non-zero idempotent otherwise.
The set of all non-zero idempotents of $A_0$ is $\{1,(-1\pm\sqrt{3}u)/2\}$ if $(i_1,j_1)=(1,1)$, while $1$ is the unique non-zero
idempotent of $A_0$ otherwise.
\end{Rk}

We will solve the classification problem of $\mathcal D_{1,1,3,3}$ as
follows. Let
$\alpha,\beta\in[0,\pi)$ and set $A=\OO_{G_{\alpha,0}^{j_1,j_2},G_{\beta,0}^{i_1,i_2}}$ and
$B=\OO_{G_{\alpha',0}^{j_1,j_2},G_{\beta',0}^{i_1,i_2}}$. Proposition \ref{P: AE} and its proof provide us with
$\xi,\eta,\xi',\eta'\in\R$ and thence algebras
$C=\OO_{F_\xi^{j_1,j_2},F_\eta^{i_1,i_2}}$ and $D=\OO_{F_{\xi'}^{j_1',j_2'},F_{\eta'}^{i_1',i_2'}}$ such that $A\simeq C$ and
$B\simeq D$. With Proposition \ref{P: Iso} below we determine all isomorphisms from $C$ to $D$. Using this we express, in Lemma
\ref{L: xieta}, the condition $C\simeq D$ in terms of $\xi$, $\eta$, $\xi'$ and $\eta'$. Using \ref{P: AE}, we express this as an
isomorphism condition on $A$ and $B$ in terms of $\alpha$, $\beta$, $\alpha'$ and $\beta'$, which enables us to produce a
classification in Theorem \ref{T: 1133}. 

\begin{Prp}\label{P: Iso} Let $(i_1,i_2),(j_1,j_2)\in\Z_2^2$ with $i_2=1$ or $j_2=1$ and let $\xi,\eta,\xi',\eta'\in\R$. If
$C=\OO_{F_\xi^{j_1,j_2},F_\eta^{i_1,i_2}}$ and 
$D=\OO_{F_{\xi'}^{j_1',j_2'},F_{\eta'}^{i_1',i_2'}}$ are in $\mathcal D_{1,1,3,3}$, and $\phi:C\to D$ is an
isomorphism, then $(i_1,j_1,i_2,j_2)=(i_1',j_1',i_2',j_2')$, and $\phi=B_c\rho$ for some $\rho\in G_2^{\langle
u\rangle }$
and $c\in\C$
satisfying
\begin{enumerate}[(i)]
\item $c^3=1$ if $i_1=j_1=1$, and
\item $c=1$, otherwise.
\end{enumerate}
\end{Prp}

Note that if $\phi=B_c\rho$ with $\rho\in G_2$, then $(L_c\rho,R_c\rho)$ is a pair of triality components of $\phi$. By virtue of
Proposition \ref{P: Ca}, the above proposition implies that $\phi: C\to D$ is an isomorphism if and only if
$(i_1,j_1,i_2,j_2)=(i_1',j_1',i_2',j_2')$, $\phi=B_c\rho$ with $\rho$ and $c$ as in the proposition, and
\begin{equation}\label{E: Classification}\begin{array}{lll}L_c\rho
F_\xi^{j_1,j_2}\rho^{-1}B_{\overline{c}}= F_{\xi'}^{j_1,j_2}&\text{and}&R_c\rho F_\eta^{i_1,i_2} \rho^{-1}B_{
\overline { c}}=F_{\eta'}^{i_1,i_2}.\end{array}
\end{equation}
(As far as Proposition \ref{P: Ca} is concerned, the right hand sides should read $(-1)^kF_{\xi'}^{j_1,j_2}$ and
$(-1)^kF_{\eta'}^{i_1,i_2}$ for some $k\in \Z_2$; however, if $k=1$, then applying either equation to 1 gives $c=-1$ or
$c^3=-1$, depending on $(i_1,j_1,i_2,j_2)$. Thus necessarily $k=0$.)
\begin{proof} Invariance of the double sign implies that
$(i_1+i_2,j_1+j_2)=(i_1'+i_2',j_1'+j_2')$. Since $\phi$ induces an isomorphism from $C_0$ to $D_0$, Remark \ref{R: A0} gives
$(i_1,j_1)=(i_1',j_1')$, proving the first part of the statement. Since $C_0=D_0=\C$ as subspaces, we have $\phi(\C)=\C $. If
$(i_1,j_1)\neq(1,1)$, then Remark \ref{R: A0} implies that $\phi(1)=1$. Remark
\ref{R: U1} then implies that $\phi\in G_2$, and by orthogonality, $\phi(\langle u\rangle )=\langle u\rangle $, whence $\phi\in
G_2^{\langle u\rangle }$. This proves the
second part in case $(i_1,j_1)\neq(1,1)$. If $(i_1,j_1)=(1,1)$, then by Proposition \ref{P: Ca} we have
\[\begin{array}{lll}\phi_1F_\alpha^{1,1}=F_{\alpha'}^{1,1}\phi&\text{and}&\phi_2F_\beta^{1,1}=F_{\beta'}^{1,1}\phi\end{array}\]
for some pair $(\phi_1,\phi_2)$ of triality components of $\phi$. Applying both both equations to 1, we get
$\phi_1(1)=\phi_2(1)=\overline{\phi(1)}$. Then $\phi_1=R_{\overline{\phi_2(1)}}\phi$ and $\phi_2=L_{\overline{\phi_1(1)}}\phi$
give
\[\begin{array}{lll}\phi_1=R_{\phi(1)}\phi&\text{and}&\phi_2=L_{\phi(1)}
\phi.\end{array}\]
By Remark \ref{R: A0} we know that $\phi(1)=e$ for some $e\in\C$ satisfying $e^3=1$ with respect to the
multiplication in $\OO$. Set $x=\phi^{-1}(v)$ and
$y=\phi^{-1}(z)$.
Then
\[\phi(xy)=\phi_1(x)\phi_2(y)=(ve)(ez)=(\overline{e}v)(z\overline{e})=\overline{e}((vz)\overline{e})=\overline{e}(e(vz))=vz,\]
where the third and fifth equalities hold since $uv'=-v'u$ for each $v'\in\C^\bot$, and the forth and last
equalities follow from
alternativity of $\OO$. Since $\phi$ maps $\C^\bot$ to itself, we in particular have $xy\bot u$. Thus $(u,x,y)$
is a Cayley triple, and there exists $\rho'\in G_2^{\langle u\rangle}$ mapping $(u,v,z)$ to $(u,x,y)$. Moreover,
$(x,ux,y,uy,xy,(ux)y)$ is a basis of $\C^\bot$. Now $e=\cos(2\pi n/3)+u\sin(2\pi n/3)$ for some $n\in\Z$, and
then
$\phi(u)=(-1)^\varepsilon(u\cos(2\pi n/3)-\sin(2\pi n/3))$ for some $\varepsilon\in\Z_2$. By computations similar to those
used to determine
$\phi(xy)$ we can determine the image of $ux$, $uy$ and $(ux)y$ in terms of $u$, $v$ and $z$. As a result we see that $\phi$ maps
\[\begin{array}{lll}(x,ux,y,uy,xy,(ux)y)&\text{to}&(v,(-1)^\varepsilon uv,z,(-1)^\varepsilon
uz,vz,(-1)^\varepsilon(uv)z).\end{array}\]
The map $\widehat\varepsilon\in G_2^{\langle u\rangle }$ was defined by $(u,v,z)\mapsto((-1)^\varepsilon u,v,z)$. By the above,
$\phi\rho'\widehat\varepsilon$ acts as identity on
$\C^\bot$, and
the matrix of its restriction to $\C$ in the basis $(1,u)$ is $\widehat{(2\pi n/3)}_0$. The same is true for
$B_d$ with $d=\cos(\pi n/3)+u\sin(\pi n/3)$. Since $B_d=B_{-d}$, we have $\phi\rho'\widehat\varepsilon=B_c$ for some
$c\in\C$ with $c^3=1$. The proof is complete upon observing that $\rho=(\rho'\widehat\varepsilon)^{-1}\in
G_2^{\langle u\rangle }$.
\end{proof}

In particular, this shows that the twelve families determined by the values of $i_1$, $i_2$, $j_1$ and $j_2$ are closed under
isomorphisms. Since $A\simeq B$ implies $A_0\simeq
B_0$, we have
\[\mathcal D_{1,1,3,3}=\coprod_{\genfrac{}{}{0pt}{}{i_1,j_1,i_2,j_2\in\Z_2}{i_2=1\vee
j_2=1}}\mathcal D_{1,1,3,3}^{i_1,j_1,i_2,j_2},\]
where $\mathcal D_{1,1,3,3}^{i_1,j_1,i_2,j_2}$ is the full subcategory of all
$A\in\mathcal D_{1,1,3,3}$ with double sign\linebreak $((-1)^{i_1+i_2},(-1)^{j_1+j_2})$ and such that
the double sign of $A_0$ is $((-1)^{i_1},(-1)^{j_1})$. Thus classifying $\mathcal D_{1,1,3,3}$ amounts to classifying each
block in the above decomposition.

The next lemma expresses the condition that \eqref{E: Classification} holds for some $\rho$ and $c$, satisfying the
conditions of Proposition \ref{P: Iso}, in terms of $\xi$, $\eta$, $\xi'$ and $\eta'$.

\begin{Lma}\label{L: xieta} Let $(i_1,i_2),(j_1,j_2)\in\Z_2^2$ with $i_2=1$ or $j_2=1$ and let $\xi,\eta,\xi',\eta'\in \R$.
\begin{enumerate}[(i)]
\item If $(i_1,j_1)\neq(1,1)$, then \eqref{E: Classification} holds for some $\rho\in G_2^{\langle u\rangle }$ and with $c=1$ if
and only if 
there exist $m,n\in\Z$ such that
\[(\xi',\eta')=\left\{ \begin{array}{ll} \pm(\xi+\pi m/3,\eta+\pi n) & \text{if\ } (i_2,j_2)=(0,1),\\
                                         \pm(\xi+\pi m,\eta+\pi n/3) & \text{if\ } (i_2,j_2)=(1,0),\\
                                         \pm(\xi+\pi m/3,\eta+(m+3n)\pi/3) & \text{if\ } (i_2,j_2)=(1,1).\end{array}\right.\]

\item If $(i_1,j_1)=(1,1)$, then \eqref{E: Classification} holds for some $\rho\in G_2^{\langle u\rangle }$ and $c\in \C$ with
$c^3=1$ if
and only if there exist $k,m,n\in\Z$ such that
\[(\xi',\eta')=\left\{ \begin{array}{ll} \pm(\xi+\pi m/3,\eta+\pi n/3) & \text{if\ } (i_2,j_2)\neq(1,1),\\
                                         \pm(\xi+\pi m/3,\eta+(2k+m+3n)\pi/3) & \text{if\ } (i_2,j_2)=(1,1).\end{array}\right.\]

\end{enumerate}
\end{Lma}

\begin{proof} Let $(i_1,i_2),(j_1,j_2)\in\Z_2^2$ with $i_2=1$ or $j_2=1$. For any $c\in\C$ satisfying (i) and
(ii) of
Lemma \ref{P: Iso} and any $\rho\in G_2^{\langle u\rangle} $, \eqref{E: Classification} asserts that
\[\begin{array}{lll}
L_c\rho C_x\sigma_u^{j_1}\sigma_W^{j_2}\rho^{-1}B_{\overline{c}}=C_{x'}\sigma_u^{j_1}\sigma_W^{j_2}&\wedge&R_c\rho
C_y\sigma_u^{i_1}\sigma_W^{i_2}\rho^{-1}B_{\overline{c}}=C_{y'}\sigma_u^{i_1}\sigma_W^{i_2}
\end{array}\]
where $x=\cos\xi+u\sin\xi$, $y=\cos\eta+u\sin\eta$, and likewise for $x'$, $y'$, $\xi'$ and $\eta'$. Since $\rho
C_t\rho^{-1}=C_{\rho(t)}$ and $\sigma_W B_t=B_t\sigma_W$
for each $t\in\langle 1,u\rangle $, this is equivalent to
\[\begin{array}{lll}
 C_{\rho(x)}\rho\sigma_u^{j_1}\sigma_W^{j_2}\rho^{-1}=L_{\overline{c}}C_{x'}\sigma_u^{j_1}B_c\sigma_W^{j_2}&\wedge&C_{\rho(y)}
\rho\sigma_u^{i_1}\sigma_W^{i_2}\rho^{-1}=R_{\overline{c}}C_{y'}\sigma_u^{i_1}B_c\sigma_W^{i_2}.
\end{array}\]
If $j_1=i_1=1$, then $\sigma_u^{j_1} B_c=B_{\overline c}\sigma_u^{j_1}$ and $\sigma_u^{i_1} B_c=B_{\overline c}\sigma_u^{i_1}$
since 
 $\sigma_u B_c=B_{\overline c}\sigma_u$. In the other cases, this holds trivially as then $c=1$ implies that
$B_c=B_{\overline c}=\Id$. Moreover, $\sigma_u$
commutes with both $\sigma_W$ and $\rho^{-1}$. Multiplying the equations from the right by $\sigma_u^{j_1+j_1}$ and
$\sigma_u^{i_1+i_2}$, respectively, and writing $\Sigma$ for $\sigma_u\sigma_W$, we therefore have
\[\begin{array}{lll}
 C_{\rho(x)}\rho\Sigma^{j_2}\rho^{-1}=L_{\overline{c}}C_{x'}B_{\overline{c}}\Sigma^{j_2}&\wedge&C_{\rho(y)
}
\rho\Sigma^{i_2}\rho^{-1}=R_{\overline{c}}C_{y'}B_{\overline{c}}\Sigma^{i_2}.
\end{array}\]
We finally collect the multiplication operators in each equation in the left hand side, and multiply
them. Using commutativity of the subalgebra $\C$ of $\OO$, and the fact that $c^3=1$, it is straight-forward to
check that the above is equivalent to
\begin{equation}\label{E: Simplified}
\begin{array}{lll}
 C_{\overline{c}\overline{x'}\rho(x)}=\Sigma^{j_2}\rho\Sigma^{j_2}\rho^{-1}&\wedge&C_{c
\overline{y'}\rho(y)}=\Sigma^{i_2}\rho\Sigma^{i_2}\rho^{-1}.\end{array}
\end{equation}
The right hand sides both belong to $G_2$, since for each $k\in\Z_2$, $\Sigma^{k}$ coincides with $\widehat
k:(u,v,z)\mapsto((-1)^ku,v,z)$. If $t\in\mathbb S^1$, then clearly $C_t=\Id$ if and only if $t^2=1$, while
$C_t=\Sigma\rho\Sigma\rho^{-1}$ for some $\rho\in G_2^{\langle u \rangle}$ if and only if $t^6=1$. (To see this, assume that
$t^6=1$. Then $(rt)^3=1$ for some $r\in\{\pm1\}$, and $C_t=C_{rt}$. Taking $s\in\C$ with $s^2=\overline{rt}$ we get
$C_t=\Sigma C_s\Sigma C_s^{-1}$. Since $s^6=1$, we can apply the fact that $w\in\mathbb S^1$ satisfies $C_w\in G_2^{\langle u
\rangle}$ if and only if $w^6=1$. Conversely, if $t^6\neq1$, then this fact implies that $C_t\notin G_2^{\langle u \rangle}$, and
in particular $C_t\neq\Sigma\rho\Sigma\rho^{-1}$ for any $\rho\in G_2^{\langle u \rangle}$.)

Returning to \eqref{E: Simplified}, we notice that $(\rho(x),\rho(y))\in\{(x,y),(\overline{x},\overline{y})\}$ for any $\rho\in
G_2^{\langle u \rangle}$. Note also that the right hand sides coincide whenever $(i_2,j_2)=(1,1)$. Altogether this gives
that
\eqref{E: Simplified} holds for \emph{some} $\rho\in G_2^{\langle u\rangle }$ and with $c=1$ 
if and only if
\[\begin{array}{cl} \left({x'}^6,{y'}^2\right)=\left(x^6,y^2\right) \vee \left({x'}^6,{y'}^2\right)=\left({\overline
x}^6,{\overline y}^2\right) & \text{if\ }(i_2,j_2)=(0,1),\\
                      \left({x'}^2,{y'}^6\right)=\left(x^2,y^6\right) \vee \left({x'}^2,{y'}^6\right)=\left({\overline
x}^2,{\overline y}^6\right) & \text{if\ }(i_2,j_2)=(1,0),\\
                      \left({x'}^6,\left(x'\overline{y'}\right)^2\right)=\left(x^6,(x\overline y)^2\right)\vee
\left({x'}^6,\left(x'\overline{y'}\right)^2\right)=\left({\overline x}^6,(\overline x y)^2\right) & \text{if\ }
(i_2,j_2)=(1,1),\end{array}\]
while \eqref{E: Simplified} holds for \emph{some} $\rho\in G_2^{\langle
u\rangle }$ and \emph{some} $c\in\C$ with $c^3=1$ if and only if
\[\begin{array}{cl} \left({x'}^6,{y'}^6\right)=\left(x^6,y^6\right)\vee \left({x'}^6,{y'}^6\right)=\left({\overline
x}^6,{\overline y}^6\right) & \text{if\ } (i_2,j_2)\neq(1,1),\\
                      \left({x'}^6,\left(x'\overline{y'}\right)^6\right)=\left(x^6,(x\overline y)^6\right)\vee
\left({x'}^6,\left(x'\overline{y'}\right)^6\right)=\left({\overline x}^6,(\overline x y)^6\right) & \text{if\ }
(i_2,j_2)=(1,1).\end{array}\]
Writing these conditions in terms
of $\xi$, $\eta$, $\xi'$ and $\eta'$ completes the proof.
\end{proof}

Expressing these conditions in terms of the parameters $\alpha$ and $\beta$ yields a classification of each $\mathcal
D_{1,1,3,3}^{i_1,j_1,i_2,j_2}$ as follows.

\begin{Thm}\label{T: 1133} Let $(i_1,i_2),(j_1,j_2)\in\Z_2^2$ with $i_2=1$ or $j_2=1$. A classification of $\mathcal
D_{1,1,3,3}^{i_1,j_1,i_2,j_2}$ is
given by
\[M=\left\{ \OO_{G_{\alpha,0}^{j_1,j_2},G_{\beta,0}^{i_1,i_2}}|(\alpha,\beta)\in
\Pi\setminus\{(j_1+j_2)\pi/2,(i_1+i_2)\pi/2\}\right\}\]
where
\[\Pi=([0,\pi/2)\times[0,\pi))\cup(\{\pi/2\}\times[0,\pi/2]).\]
\end{Thm}

\begin{proof} Let $\alpha,\beta,\alpha',\beta'\in[0,\pi)$ and let $(\xi,\eta)$ and 
$(\xi',\eta')$ be obtained from $(\alpha,\beta)$ and $(\alpha',\beta')$, respectively, by equations \eqref{E: i1j1} and
\eqref{E: i2j2}, with $\gamma=\gamma'=0$. In the case $(i_1,j_1,i_2,j_2)=(0,0,0,1)$ we thus have
\[\begin{array}{lll}
 (\xi,\eta)=\left(\frac{\alpha+2\beta}{3},-\beta\right)&\text{and}&(\xi',\eta')=\left(\frac{\alpha'+2\beta'}{3},-\beta'\right)
\end{array}\]
and similar expressions hold for the other cases. Using this, we write the conditions in Lemma \ref{L: xieta} in
terms of $(\alpha,\beta)$ and $(\alpha',\beta')$. Each condition then in fact turns out to be equivalent to
\begin{equation}\label{E: pi}
\begin{array}{lll}(\alpha',\beta')=(\alpha,\beta) &\vee& (\alpha',\beta')=(\pi-\alpha,\pi-\beta).\end{array}
\end{equation}
Thus if 
\[\OO_{G_{\alpha,0}^{j_1,j_2},G_{\beta,0}^{i_1,i_2}},\OO_{G_{\alpha',0}^{j_1,j_2},G_{\beta',0}^{i_1,i_2}}\in\mathcal
D_{1,1,3,3}^{i_1,j_1,i_2,j_2},\]
then they
are isomorphic if and only if \eqref{E: pi} holds, whence $M$ is
irredundant. Since each $A\in\mathcal D_{1,1,3,3}^{i_1,j_1,i_2,j_2}$ is isomorphic to
$\OO_{G_{\alpha,0}^{j_1,j_2},G_{\beta,0}^{i_1,i_2}}$ for some
$(\alpha,\beta)\in[0,\pi)^2$ different from $((j_1+j_2)\pi/2,(i_1+i_2)\pi/2)$, it follows that $M$ is exhaustive as well.
\end{proof}

\begin{Rk} The theorem and its proof imply that if
\[A=\OO_{G_{\alpha,0}^{j_1,j_2},G_{\beta,0}^{i_1,i_2}}\simeq \OO_{G_{\alpha',0}^{j_1,j_2},G_{\beta',0}^{i_1,i_2}}=B,\]
then $\widehat\varepsilon:A\to B$ is an isomorphism for some $\varepsilon\in\Z_2$, and conversely. Thus \emph{a posteriori} we
get a quasi-description of $\mathcal D_{1,1,3,3}^{i_1,j_1,i_2,j_2}$ as follows. Let $j=j_1+j_2$ and $i=i_1+i_2$ with $i_2=1$ or
$j_2=1$. The group $\Z_2$ acts on the set $\left(\mathbb S^1\times\mathbb S^1\right)_{ij}$ defined in Proposition \ref{P: 1,1,6}
by
\begin{equation}\label{E: epsilon}
\varepsilon\cdot (a,b)=(K^\varepsilon(a),K^\varepsilon(b)),
\end{equation}
and the map
\[\phantom{}_{\Z_2}(\mathbb S^1\times \mathbb S^1)_{ij}\to \mathcal
D_{1,1,3,3}^{i_1,j_1,i_2,j_2}\]
defined on objects by mapping $(a,b)=(\cos2\alpha+u\sin2\alpha,\cos2\beta+u\sin2\beta)$ with $\alpha,\beta\in[0,\pi)$
to
$\OO_{G_{\alpha,0}^{j_1,j_2},G_{\beta,0}^{i_1,i_2}}$, and on morphisms by $\varepsilon\mapsto\widehat\varepsilon$, is a dense
functor which detects non-isomorphic objects. However, we do not know how, if at all, this can be deduced \emph{a priori}.

Let now $i_2=j_2=0$ and set $i=i_1$ and $j=j_1$. Then the action of $Z_2$ on $(\mathbb S^1\times \mathbb S^1)_{ij}$ induced by
\eqref{E: Action0} coincides with the action defined in \eqref{E: epsilon}, and it follows from Proposition \ref{P: 1,1,6} that
the restriction of the functor $\mathcal F^{ij}$ defined there to
$\phantom{}_{\Z_2}(\mathbb S^1\times \mathbb S^1)_{ij}$ is a
quasi-description of $\mathcal D_{1,1,6}^{ij}$. With our observation earlier in this section that
$G_{\theta,0}^{k,0}=\lambda_t^{(k)}$, with $t=\cos 2\theta+u\sin 2\theta$, for each $\theta\in\R$ and $k\in\Z_2$, this brings the
treatment of $\mathcal D_{1,1,6}$ and $\mathcal D_{1,1,3,3}$ on
an equal footing.
\end{Rk}

\section{Classification of $\mathcal L_0$ and $\mathcal H_0$}
We will now classify $\phantom{}_{SO_3}\left(\T\times\T\right)$ and $\phantom{}_{SO_3}\left[\T\times\T\right]^2$ up to
isomorphism, by finding 
transversals for the actions of $SO_3$ on $\T\times\T$ and $\left[\T\times\T\right]^2$ defined in Section 4.1. The
classification of the first groupoid gives, via the functors $\mathcal J^{ij}$ from Section 4.2 , a classification of the
subcategory
$\mathcal L_0\subset \mathcal D$. The classification of the second becomes, after removing a finite set of objects, a
classification of the category $\phantom{}_{SO_3}S$ from Section 4.3, which in turn gives a classification of $\mathcal H_0\subset
\mathcal D$ via the functors $\mathcal K_*^{ij}$ introduced in that section.

Recall that we identify $\mathbb S(\HH)$ with $\T$, and fix an orthonormal pair $(u,v)$ in $\HH\cap1^\bot$. We now define
the sets
\[P=\{\cos\alpha+(u\cos\beta+v\sin\beta)\sin\alpha|\alpha\in (0,\pi)\wedge\beta\in[0,\pi]\},\]
and
\[P_0=\{\cos\alpha+u\sin\alpha|\alpha\in (0,\pi)\},\]
and, for each $A\subseteq (0,\pi)$ and $B\subseteq [0,\pi]$, we set 
\[P_{AB}=\{\cos\alpha+(u\cos\beta+v\sin\beta)\sin\alpha|\alpha\in A \wedge \beta\in B\}.\]

\begin{Prp}\label{P: T2} A transversal for the action of $SO_3$ on $\T\times\T$ is given by
\[M=(\{\pm1\}\times\{\pm1\}) \cup(\{\pm1\}\times P_0) \cup (P_0\times\{\pm1\})\cup (P_0\times P).\]
\end{Prp}

\begin{proof} First we show that $M$ exhausts the orbits. Take any $(a,b)\in\T\times\T$. If $a\in\langle 1\rangle$,
then
$(a,b)=(\pm1,\cos\alpha+w\sin\alpha)$ for some $\alpha\in[0,\pi]$ and $w\in\mathbb S\left(1^\bot\right)$.
Since
$SO_3$ acts transitively on $\mathbb S\left(1^\bot\right)$, there exists $q\in\T$ such that $qw\overline{q}=u$, and then
\[\kappa_q\cdot(a,b)\in(\{\pm1\}\times\{\pm1\}) \cup(\{\pm1\}\times P_0).\] 
If $a\notin\langle 1\rangle$, then $a=\cos\alpha+w\sin\alpha$ for some $\alpha\in(0,\pi)$ and $w\in\mathbb S\left(1^\bot\right)$.
As above
there exists
$q\in\T$ such that $\kappa_q\cdot(a,b)=(\cos\alpha+u\sin\alpha,b')$, for some $b'\in \T$. Then
$b'=\cos\alpha'+(u\cos\beta'+w'\sin\beta')\sin\alpha'$ for some $\alpha'\in(0,\pi)$, $\beta'\in[0,\pi]$ and
$w'\in\mathbb S\left(\langle 1,u\rangle^\bot\right)$. Moreover, the subgroup of $SO_3$ consisting of all elements fixing $u$ acts
transitively on $\langle v,uv\rangle$ and
thus there exists $q'\in\T$ such that $q'w'\overline{q'}=v$.
Altogether $\kappa_{q'q}\cdot(a,b)\in(P_0\times\{\pm1\})\cup (P_0\times P)$. Thus $M$ is exhaustive. To prove irredundancy, we
recall that conjugation by a unit quaternion fixes $1\in \HH$.
Thus if $\alpha,\alpha'\in[0,\pi]$ satisfy $\kappa_q(\cos\alpha+u\sin\alpha)=\cos\alpha'+u\sin\alpha'$ for some $q\in\T$, then
$\alpha=\alpha'$. This implies that no orbit intersects more than one of the four sets
$\{\pm1\}\times\{\pm1\}$, $\{\pm1\}\times P_0$, $P_0\times\{\pm1\}$ and $P_0\times P$, and that the first three of these sets are
irredundant. It moreover proves that if $\kappa_q\cdot(a,b)=(a',b')$ for some $q\in\T$ and $(a,b),(a',b')\in P_0\times P$, then
$\kappa_q(a)=a'$ gives $a=a'$ and $\kappa_q(u)=u$, and by coordinatewise comparison of $b$ and $b'$ one concludes that they
coincide. Thus $P_0\times P$ is irredundant.
\end{proof}

Thus we have classified $\phantom{}_{SO_3}\left(\T\times\T\right)$, and a classification of $\mathcal L_0^{ij}$ is obtained for
each $(i,j)\in\Z_2^2$ by applying the functor $\mathcal J^{ij}$ of Theorem \ref{T: Quasi1} to this classification.

In order to classify $\phantom{}_{SO_3}\left[\T\times\T\right]^2$, we will first obtain a transversal for
the action of $SO_3$ on $\left[\T\times\T\right]$, and pass to the action of $SO_3$ on $\left[\T\times\T\right]^2$. The strategy
is detailed in the following
lemma, for which we set
\begin{equation}\label{E: M}
M_1=(\{1\}\times\{\pm1\}) \cup(\{1\}\times P_0) \cup (P_0\times\{1\})\cup (P_0\times P)^+,
\end{equation}
with 
\[(P_0\times P)^+=\{(a,b)\in P_0\times P|(\Re(a),\Re(b))\succeq(0,0)\},\]
where $\succeq$ is the relation ``greater than or equal to'' with respect to the lexicographic order on $\R^2$, and $\Re$ denotes
the orthogonal projection onto $\langle 1\rangle$.

\begin{Lma}\label{L: Generalities} Let $M_1$ be as in \eqref{E: M}.
\begin{enumerate}[(i)]
\item The quotient map $\T\times\T\to\left[\T\times\T\right]$ maps $M_1$ bijectively onto $[M_1]$, and $[M_1]$ is a transversal of
the action
of $SO_3$ on $\left[\T\times\T\right]$.
\item The set $[M_1]\times [\T\times \T]$ exhausts the orbits of the action of $SO_3$ on $\left[\T\times\T\right]^2$.
\item Two elements $([a_1,b_1],[a_2,b_2])$ and $([c_1,d_1],[c_2,d_2])$ in $[M_1]\times \left[\T\times\T\right]$ are in the same
orbit if and
only if $[a_1,b_1]=[c_1,d_1]$ and $\kappa_q\cdot[a_2,b_2]=[c_2,d_2]$ for some $q\in \T$ with $\kappa_q\in\Stab([a_1,b_1])$.
\end{enumerate}
\end{Lma}

\begin{proof} The bijectivity statement in (i) follows from the fact that $M_1$ does not contain the negative of any of its
elements. To show that $[M_1]$ is exhaustive, we note that Proposition \ref{P: T2} implies that $[M]$ indeed is. Let thus
$(a,b)\in M$. By construction, $M_1$ contains either $(a,b)$ or $(-\overline{a},-\overline{b})$. Since $SO_3$ acts transitively on
$\mathbb S\left(1^\bot\right)$, the elements $[a,b]$ and $[\overline{a},\overline{b}]=[-\overline{a},-\overline{b}]$ are in the
same orbit of
the action of $SO_3$ on $\left[\T\times\T\right]$. Thus $[M_1]$ is exhaustive. 

To show that it is irredundant, assume that for some $(a,b)$ and
$(c,d)$ in $M_1$, $[a,b]$ and $[c,d]$ are in the same orbit. Then there exist $\varepsilon\in\Z_2$ and $q\in \T$ such that
$(\kappa_q(a),\kappa_q(b))=((-1)^\varepsilon c,(-1)^\varepsilon d)$. If one of $a$, $b$, $c$ and $d$ is not orthogonal to 1, then
either $\Re(a)$ and $\Re(c)$ are positive, or $\Re(b)$ and $\Re(d)$ are positive. In both cases $\varepsilon=0$, which implies
that
$(a,b)=(c,d)$ since $M_1\subset M$ is irredundant with respect to the action of $SO_3$ on $\T\times\T$. If all of $a$, $b$, $c$
and $d$ are orthogonal to 1, then $(a,b)=(u,u\cos\beta+v\sin\beta)$ and $(c,d)=(u,u\cos\beta'+v\sin\beta')$, for some
$\beta,\beta'\in[0,\pi]$. Then $(\kappa_q(a),\kappa_q(b))=((-1)^\varepsilon c,(-1)^\varepsilon d)$ implies that 
\[\begin{array}{lll} \kappa_q(u)=(-1)^\varepsilon u &\text{and} & \kappa_q(u)\cos\beta + \kappa_q(v)\sin\beta= (-1)^\varepsilon
u\cos\beta'+(-1)^\varepsilon\sin\beta',\end{array}\]
and substituting the first equation into the second gives $\cos\beta=\cos\beta'$, whence
$(a,b)=(c,d)$. This proves irredundancy and completes the proof of (i), and (ii) and (iii) follow.
\end{proof}

In view of the above lemma, two tasks remain in order to obtain a transversal for the action of $SO_3$ on
$\left[\T\times\T\right]^2$. The
first is to determine the stabilizers of the elements in $[M_1]$ with respect to the action of $SO_3$ on
$\left[\T\times\T\right]$.
Secondly, we must compute, for each such stabilizer $G$, a transversal for the restriction to $G\leq SO_3$ of the action of $SO_3$
on $\left[\T\times\T\right]$. These tasks are carried out in the following two lemmata.

\begin{Lma} Let $M_1$ be as in \eqref{E: M}, and let $(a,b)\in M_1$.
\begin{enumerate}[(i)]
\item If $a,b\in \langle 1\rangle$, then $\Stab([a,b])=SO_3$.
\item If $a,b\in\langle 1,u\rangle$ and neither both belong to $\langle 1\rangle$ nor both belong to $\langle u\rangle$, then
$\Stab([a,b])=\{\kappa_q|q\in\mathbb
S(\langle 1,u\rangle)\}$.
\item If $a,b\in \langle u\rangle$, then $\Stab([a,b])=\{\kappa_q|q\in\mathbb S(\langle 1,u\rangle\cup\langle
v,uv\rangle)\}$.
\item If $a,b\in \langle u,v\rangle$ and are linearly independent, then $\Stab([a,b])=\{\Id,\kappa_{uv}\}$.
\item Otherwise, $\Stab([a,b])$ is trivial.
\end{enumerate}
\end{Lma}

\begin{proof} By definition of the action, $\kappa_q\in\Stab([a,b])$ if and only if 
\[(\kappa_q(a),\kappa_q(b))=((-1)^\varepsilon a,(-1)^\varepsilon b)\]
for some $\varepsilon\in\Z_2$, i.e.\ if and only if $q$
commutes with both $a$ and $b$, or anticommutes with them. For each
$w\in\HH$, the centralizer of $w$ is $\HH$ if $w\in\langle 1\rangle=Z(\HH)$, and $\langle 1,w\rangle$
otherwise, while the set of all quaternions anticommuting with $w$ is empty if $\Re(w)\neq 0$, and $\langle 1,w\rangle^\bot$
otherwise. The assertions follow.
\end{proof}

For the next lemma, we introduce, for any $m,n\in\{1,2,3,4\}$, the notation
\[\T_{mn}=\{r_1+r_2u+r_3v+r_4uv\in\T|(r_m,r_n)\succeq(0,0)\},\]
where $\succeq$ refers to the lexicographic order.

\begin{Lma}\label{L: Stabilizers} Let $M_1$ be as in \eqref{E: M}.
\begin{enumerate}[(i)]
\item A transversal for the action of $SO_3$ on $\left[\T\times\T\right]$ is given by $[M_1]$.

\item A transversal for the restriction of the action to $\{\kappa_q|q\in\mathbb S(\langle 1,u\rangle)\}$ is given by $[M_2]$,
where $M_2$ is the union of
\[\begin{array}{llll}(\{1\}\cup P_0)\times (\{\pm1\}\cup P),&\{v\}\times
\T_{12},&\text{and}&\left(P_{\left(0,\frac{\pi}{2}\right)(0,\pi)}\cup
P_{\{\frac{\pi}{2}\}\left(0,\frac{\pi}{2}\right)}\right)\times \T.\end{array}\]

\item A transversal for the restriction of the action to $\{\kappa_q|q\in\mathbb S(\langle 1,u\rangle\cup\langle v,uv\rangle)\}$
is given by $[M_3]$, where $M_3$ is the union of
\[\begin{array}{ll}\{1\}\times \left(\{\pm1\}\cup P_{(0,\pi)\left[0,\frac{\pi}{2}\right]}\right),&
P_{\left(0,\frac{\pi}{2}\right)\{0\}}\times (\{\pm1\}\cup P),\end{array}\]
\[\begin{array}{ll}\{u\}\times\left(\{1\}\cup P_{\left(0,\frac{\pi}{2}\right]\left[0,\pi\right]}\right),&
\{v\}\times\left(\T_{14}\cap\T_{24}\right),\end{array}\]
\[\begin{array}{llll}
P_{\left\{\frac{\pi}{2}\right\}\left(0,\frac{\pi}{2}\right)}\times \T_{14},&
P_{\left(0,\frac{\pi}{2}\right)\left\{\frac{\pi}{2}\right\}}\times \T_{24},&
\text{and}&
P_{\left(0,\frac{\pi}{2}\right)\left(0,\frac{\pi}{2}\right)}\times \T.\end{array}\]

\item A transversal for the restriction of the action to $\{\Id,\kappa_{uv}\}$ is given by
$[M_4]$, where
\[M_4=\left(\T_{1,4}\cap\T_{2,3}\right)\times \T.\]
\end{enumerate}
Moreover, for each $1\leq k\leq 4$, the quotient map $\T\times\T\to\left[\T\times\T\right]$ maps $M_k$ bijectively onto $[M_k]$.
\end{Lma}

Recall that for each $a,b,c,d\in\T$, $[a,b]$ and $[c,d]$ belong to the same orbit of the restriction of the action to $G\leq
SO_3$ if and only if $(c,d)=((-1)^\varepsilon \kappa_q(a),(-1)^\varepsilon \kappa_q(b))$ for some $\varepsilon\in\Z_2$ and
$q\in\T$ with
$\kappa_q\in G$.

\begin{proof} Statement (i) follows immediately from Lemma \ref{L: Generalities}(i).

For statement (ii) we note that $\{\kappa_q|q\in\mathbb S(\langle 1,u\rangle)\}$ fixes $1$ and $u$ and acts transitively on
$\mathbb S(\langle v,uv\rangle)$. Let now $(a,b)\in\T\times\T$. If $a\in \langle 1,u\rangle$, then $(-1)^\varepsilon a\in\{1\}\cup
P_0$ for some $\varepsilon\in \Z_2$, and there exists $q\in\mathbb S(\langle 1,u\rangle)$ such that $\kappa_q$ maps $b$ into
$\{\pm1\}\cup P$ and fixes $a$. If $a\notin\langle 1,u\rangle$, then there exists $q\in\mathbb S(\langle 1,u\rangle)$ such that
the projection of $\kappa_q(a)$ onto $\langle v,uv\rangle$ is $(-1)^\varepsilon v$, where $\varepsilon$ is determined as follows.
If $a\bot\langle 1,u\rangle$, then $\varepsilon$ is chosen so that $(-1)^\varepsilon b\in\T_{1,2}$. If not, it can be chosen so
that  $(-1)^\varepsilon a\in\T_{1,2}$. In both cases we get $((-1)^\varepsilon \kappa_q(a),(-1)^\varepsilon \kappa_q(b))\in M_2$.
Altogether we have shown that $[M_2]$ exhausts the orbits of the action.

To prove that $[M_2]$ is irredundant, assume that $(a,b)$ and $(c,d)$ in $M_2$ satisfy $(c,d)=((-1)^\varepsilon
\kappa_q(a),(-1)^\varepsilon \kappa_q(b))$ for some $\varepsilon\in\Z_2$ and $q\in\mathbb S(\langle 1,u\rangle)$. Since
conjugation by
$q$
fixes $1$ and $u$ and maps $\langle v,uv\rangle$ to itself, it follows that $(a,b)$ and $(c,d)$ belong to the same one of the
three
constituents\footnote{The \emph{constituents} of each $M_k$ are the subsets of $\T\times\T$ as a union of which $M_k$ is
explicitly constructed in the above lemma. Thus $M_2$ has three constituents, and $M_3$ has seven.} of $M_2$. It moreover follows
that
\begin{equation}\label{E: kappa}
(x,y\in\{\pm1\}\cup P\wedge\kappa_q(x)=y)\Longrightarrow x=y.
\end{equation}
Now if $(a,b)$ and $(c,d)$ lie in $(\{1\}\cup P_0)\times (\{\pm1\}\cup P)$,
then $q$ commutes with $a$, implying that $\kappa_q(a)=a$, and then $\varepsilon=0$ as $-a\notin\{1\}\cup P_0$. Then $a=c$, and
$b,d\in \{\pm1\}\cup P$
with $\kappa_q(b)=d$, which by \eqref{E: kappa} implies that $b=d$. If $(a,b)$ and $(c,d)$ lie in $\{v\}\times \T_{1,2}$, then
$(-1)^\varepsilon \kappa_q(v)=v$, implying that either $\varepsilon=0$ and $\kappa_q=\Id$ or $\varepsilon=1$ and
$\kappa_q=\kappa_u$. If the projection of $b$ onto $\langle 1,u\rangle$ is non-zero, then the first possibility must hold; if not,
then $-\kappa_u(b)=b$. In both cases $(a,b)=(c,d)$. Finally let $(a,b)$ and $(c,d)$ be in the rightmost constituent of $M_2$. Then
each of $a$ and $c$ either has a positive $1$-coordinate, or the $1$-coordinate is zero and the $u$-coordinate is positive. Since
$\kappa_q$ fixes $1$ and $u$ this implies that $\varepsilon=0$, whence $a=c$ by \eqref{E: kappa}. Since $a$ has a non-zero
$v$-coordinate we have $\kappa_q(v)=v$, which, in view of $q\in\langle 1,u\rangle$ implies that $\kappa_q=\Id$. Therefore
$(a,b)=(c,d)$, completing the proof of (ii).

As for statement (iii), observe that the group acting here contains as a subgroup that of item (ii), whence $M_2$ exhausts
the orbits.
The set $M_3$ is a subset of $M_2$, and by computations similar to those in the previous item one sees that for each $(a,b)\in
M_2$, either $(a,b)\in M_3$ or there exists $(c,d)\in M_3$ with $(c,d)=((-1)^\varepsilon \kappa_q(a),(-1)^\varepsilon
\kappa_q(b))$ for some
$\varepsilon\in\Z_2$ and $q\in \mathbb S(\langle v,uv\rangle)$. To prove that $M_3$ is irredundant, assume that
$(c,d)=((-1)^\varepsilon \kappa_q(a),(-1)^\varepsilon \kappa_q(b))$ for some $(a,b)\neq (c,d)\in M_3$, $\varepsilon\in \Z_2$ and
$q\in \mathbb S(\langle 1,u\rangle\cup\langle v,uv\rangle)$. Then $q\in\mathbb S(\langle v,uv\rangle)$ as $M_3$ is a subset of
$M_2$, and $[M_2]$ is irredundant with respect to the restriction of the action to $\{\kappa_q|q\in\mathbb S(\langle
1,u\rangle)\}$, by the previous item. Thus the norm of the projection of $a$ onto each of $\langle 1\rangle$, $\langle u\rangle$
and $\langle v,uv\rangle$ coincides with that of $c$, which upon inspection implies that $(a,b)$ and $(c,d)$ belong
to the same constituent of $M_3$. Checking that no such $(a,b),(c,d)$ exists is then done separately in each constituent, using
computations similar to those of the previous item. 

Statement (iv) is easy to check. Finally the
statement that each $M_k$ is mapped bijectively onto $[M_k]$ follows from the fact that for each $k$,
$M_k$ does not contain the negative of any of its elements.
\end{proof}

Combining the above two lemmata, we immediately arrive at a transversal for the action of $SO_3$ on $\left[\T\times\T\right]^2$ as
follows.

\begin{Cor} Let $M_1,\ldots,M_4$ be as in Lemma \ref{L: Stabilizers}. A transversal for the action of $SO_3$ on
$\left[\T\times\T\right]^2$ is given by the set $N$ of all $([a_1,b_1],[a_2,b_2])\in
[M_1]\times\left[\T\times\T\right]$ satisfying any of the following mutually exclusive conditions.
\begin{enumerate}[(i)]
\item $a_1,b_1\in \langle 1\rangle$ and $(a_2,b_2)\in M_1$.
\item $a_1,b_1\in\langle 1,u\rangle$ and are neither both in $\langle 1\rangle$ nor both in $\langle u\rangle$, and 
$(a_2,b_2)\in M_2$.
\item $a_1,b_1\in \langle u\rangle$ and $(a_2,b_2)\in M_3$.
\item $a_1,b_1\in \langle u,v\rangle$ and are linearly independent, and $(a_2,b_2)\in M_4$.
\item $a_1,b_1$ satisfy none of the above, and $a_2,b_2\in\T$.
\end{enumerate}
\end{Cor}

\begin{Rk} The four elements $([1,\pm1],[1,\pm1])$ satisfy the first item above. The set
$N'=N\setminus\{([1,\pm1],[1,\pm1])\}$ is thus a transversal for the action of $SO_3$ on the set $S$ from Remark \ref{R:
Remark}. By virtue of Theorem \ref{T: Quasi2} and using the functors $\mathcal K_*^{ij}$ defined therein, a classification of
the category $\mathcal H_0^{ij}$ is given by $\mathcal K_*^{ij}(N')$ for each $(i,j)\in\Z_2^2$.
\end{Rk}

This concludes the classification of $\mathcal D$.

\bibliographystyle{amsplain}
\bibliography{references}
\end{document}